\documentclass[11pt, a4paper]{article}

\title{Quasi-invariance for $\SLE$ welding measures}
\date{}

\usepackage[T1]{fontenc}

\usepackage[utf8]{inputenc}

\usepackage{mathtools}
\usepackage{lmodern}
\usepackage{amsthm,amsmath,amsfonts,amssymb}
\usepackage{graphicx}
\usepackage{enumerate,enumitem}
\usepackage{authblk}
\usepackage{cite,url}
\usepackage[normalem]{ulem}
\usepackage{mathrsfs}
\usepackage{verbatim}
\usepackage{multirow}
\usepackage{array}
\newcolumntype{P}[1]{>{\centering\arraybackslash}p{#1}}
\usepackage{longtable,booktabs}

\usepackage{esint}

\usepackage{hyperref}
\hypersetup{
    colorlinks=true, % make the links colored
    linkcolor=blue, % color TOC links in blue
    urlcolor=black, % color URLs in black
    citecolor=black,
    linktoc=all % 'all' will create links for everything in the TOC
}

\usepackage[font=normal,labelfont=bf,labelsep=quad]{caption}

\allowdisplaybreaks

\setlist[enumerate]{topsep = 1ex, leftmargin=1cm, itemsep= -2pt}
\setlist[itemize]{topsep = 1ex, leftmargin=1cm, itemsep= -2pt}

\let\OLDthebibliography\thebibliography
\renewcommand\thebibliography[1]{
  \OLDthebibliography{#1}
  \setlength{\parskip}{1pt}
  \setlength{\itemsep}{2pt}
}

\newtheorem{thm}{Theorem}[section]
\newtheorem{cor}[thm]{Corollary}

\newtheorem{lem}[thm]{Lemma}
\newtheorem{prop}[thm]{Proposition}

\theoremstyle{definition} 
\newtheorem{df}[thm]{Definition}

\newtheorem{remark}[thm]{Remark}

\numberwithin{equation}{section}

\usepackage{floatrow}

\usepackage{mathtools}

\usepackage[dvipsnames]{xcolor}

\global\long\def\ii{\mathfrak{i}}
\renewcommand{\liminf}{\varliminf}

\newcommand{\abs}[1]{\left\lvert #1 \right \rvert}

\newcommand{\norm}[1]{\lVert #1 \rVert}

\newcommand{\mc}[1]{\mathcal{#1}}
\newcommand{\m}[1]{\mathbb{#1}}

\newcommand{\ms}[1]{\mathscr{#1}}
\newcommand{\mr}[1]{\mathrm{#1}}

\renewcommand\Re{\operatorname{Re}}

\def\X{\Gamma}
\def\XX{\Gamma}

\def\CR{\operatorname{RM}}

\def \Un {\operatorname{U}}
\def\SLE{\operatorname{SLE}}

\def \Homeo {\operatorname{Homeo}}
\def\mob{\mathrm{M\ddot ob}}
\def\Diff{\operatorname{Diff}}

\def\QS{\operatorname{QS}}
\def\Sy{\operatorname{S}}

\def\Po{\operatorname{P}}

\def\WP{\operatorname{WP}}

\def \WP {\operatorname{WP}}
\def \HS {\operatorname{HS}}
\def \QS {\operatorname{QS}}
\def \COV {\operatorname{Cov}}
\def \TR {\operatorname{Tr}}
\def \SP {\operatorname{Sp}}
\def\a{\alpha}
\def\b{\beta}
\def\g{\gamma}
\def\G{\Gamma}
\def\d{\delta}

\def\e{\epsilon}

\def\t{\theta}
\def\T{\Theta}

\def\l{\lambda}
\def\k{\kappa}

\def\s{\sigma}

\def\o{\omega}
\def\O{\Omega}
\def\vare{\varepsilon}

\def\Chat{\hat{\m{C}}}

\def\dd{\mathrm{d}}

\def\var{\mathrm{Var}}

\def\1{\mathbf{1}}

\author{Shuo Fan\thanks{\protect\url{shuofan.math@gmail.com} Tsinghua University, Beijing, China; IHES, Bures-sur-Yvette, France}  \qquad Jinwoo Sung\thanks{\protect\url{jsung@math.uchicago.edu} University of Chicago, Chicago, IL, USA}}

\begin{document}

\maketitle
\begin{abstract}
A large class of Jordan curves on the Riemann sphere can be encoded by circle homeomorphisms via conformal welding, among which we consider the welding homeomorphism of the random SLE loops and the Weil--Petersson class of quasicircles. 
It is known from the work of Carfagnini and Wang \cite{CW24} that the Onsager--Machlup action functional of SLE loop measures --- the Loewner energy --- coincides with the K\"ahler potential of the unique right-invariant K\"ahler metric on the group of Weil--Petersson circle homeomorphisms. This identity suggests that the group structure given by the composition shall play a prominent role in the law of SLE welding, which is so far little understood.

 In this paper, we show a Cameron--Martin type result for random weldings arising from Gaussian multiplicative chaos, especially the SLE welding measures, with respect to the natural group action by Weil--Petersson circle homeomorphisms. More precisely, we show that these welding measures are quasi-invariant when pre- or post-composing the random welding by a fixed Weil--Petersson circle homeomorphism. Our proof is based on the characterization of the composition action in terms of Hilbert--Schmidt operators on the Cameron--Martin space of the log-correlated Gaussian field and the description of the SLE welding as the welding of two independent Liouville quantum gravity (LQG) disks. 
\end{abstract}

\tableofcontents
\section{Introduction}
\subsection{Background and motivation}
The Schramm--Loewner evolution (SLE$_\kappa$) loop measure is a one-parameter family of $\s$-finite measures indexed by $\kappa \in (0,8)$ on the space of non-self-crossing loops on the Riemann sphere. For $\k \in (0,4]$, the SLE$_\k$ loop measure is supported on Jordan curves. SLE curves first arose as interfaces in the scaling limits of critical lattice models \cite{Sch00,SW01,LSW04LERWUST,SS06-6,SS09-4,CDCH,ang2024sleloopmeasureliouville} and the loop version of SLE was constructed and studied in \cite{KS07,werner_measure,MR3,Benoist_loop,zhan2020sleloop}.
    
    The group of Weil--Petersson homeomorphisms, denoted $\WP(\m S^1)$, is a subgroup of quasisymmetric circle homeomorphisms introduced in \cite{cui00,TT06}. The most straightforward characterization of $\WP(\m S^1)$ is that $\varphi \in \WP(\m S^1)$ if and only if $\varphi$ is absolutely continuous and $\log \abs{\varphi'}$ belongs to the fractional Sobolev space $H^{1/2}(\m S^1)$ \cite{Shen18}. Takhtajan and Teo \cite{TT06} showed that the Weil--Petersson Teichm\"uller space $T_0(1)$, which is the quotient space of $\WP(\m S^1)$ up to post-composition by M\"obius transformations, carries an essentially unique K\"ahler structure that is right-invariant (namely, invariant under post-composition by elements in $\WP(\m S^1)$). 

A close relationship between these two objects was first observed by Yilin Wang \cite{W2}, who showed the equivalence between the Loewner energy  --- the large deviation rate function for SLE$_\k$ as $\k\to 0$ \cite{W1, RW, PW24_LDP_multichordal} --- and the universal Liouville action --- the unique right-invariant K\"ahler potential on the Weil--Petersson Teichm\"uller space $T_0(1)$ found by Takhtajan and Teo \cite{TT06}. Recently, the Loewner energy was further proven to be the Onsager--Machlup functional of the SLE$_\k$ loop measure for $\k \in (0,4]$ in the work of Carfagnini and Wang \cite{CW24}, thus extending the connection beyond the semiclassical regime. Hence, it is natural to wonder how the K\"ahler and group structures on the Weil--Petersson Teichm\"uller space manifest on the SLE$_\k$ loop measure. In this work, we concentrate on the latter part of the question: i.e., quasi-invariance under the composition of circle homeomorphisms given by conformal welding.

The Loewner energy of a Jordan curve is defined as the Dirichlet energy of its driving function \cite{W1,RW,SungWang}. The driving function of an SLE$_\k$ loop is $\sqrt \k$ times the two-sided Brownian motion, whose Cameron--Martin space consists of functions with finite Dirichlet energy. Thus, the loops with finite Loewner energy, which are Weil--Petersson quasicircles as identified in \cite{W2}, can be viewed as the Cameron--Martin space of SLE$_\k$ loop measures for the ``addition'' of driving functions.  
Our result shows that despite the nonlinearity of the composition of welding homeomorphisms as opposed to the addition of driving functions, $\WP(\m S^1)$ behaves as the Cameron--Martin space of the random welding homeomorphism corresponding to SLE$_\k$ loops under the natural group action given by composition.

\subsection{Main result}
To describe our main result, we give a brief overview of conformal welding.
Given an oriented Jordan curve $\eta$ on the Riemann sphere $\Chat$, let $f : \m D \rightarrow \Omega$ and $g : \m D^* \rightarrow \Omega^*$ denote any conformal maps from the unit disk $\m D = \{z \in \Chat: |z|<1\}$ and the outer unit disk $\m D^* := \Chat \setminus \overline{\m D}$ onto the bounded and unbounded components $\Omega$ and $\Omega^*$ of $\Chat \setminus \eta$, respectively. A classic result of Carath\'eodory states that $f$ and $g$ extend to homeomorphisms on the closed disks $\overline{\m D}$ and $\overline{\m D^*}$, respectively. The map $\psi:=(g^{-1} \circ f)|_{\m S^1}$ is then an orientation-preserving homeomorphism on the unit circle $\m S^1$ to itself. We call any homeomorphism that arises in this way a \textit{welding}. 

Note that for any Jordan curve $\eta$ and a M\"obius map $\o \in \mob (\Chat)$ of the Riemann sphere $\Chat$, the image $\o(\eta)$ has the same welding as $\eta$. On the other hand, we may pre-compose them by M\"obius maps fixing $\m S^1$, the space of which we denote $\mob(\m S^1)$. Hence, conformal welding is better described in terms of the correspondence 
    \begin{equation}\label{Int eq:corres}
    \mob (\Chat)\backslash\{\text{Oriented~Jordan~curves}\} \rightarrow
      \mob(\m S^1)\backslash\Homeo_+(\m S^1)/\mob(\m S^1)  .
    \end{equation}
    This correspondence is neither injective nor onto; see \cite{Bishop07} and the references therein. (We note a recent article \cite{Rodriguez-two-weldings} that describes every circle homeomorphism as a composition of two welding homeomorphisms.) The map \eqref{Int eq:corres} is injective when restricting to \textit{conformally removable curves} (see Section \ref{section:removability}). For example, quasicircles (images of the unit circle under quasiconformal homeomorphisms of the Riemann sphere $\Chat$) are conformally removable \cite{BA56}, and SLE$_\kappa$ curves are almost surely conformally removable for $\kappa \in (0,\k_0)$ where $\k_0\in(4,8)$ \cite{JS-removability,Rohde_Schramm,kavvadias2022conformal,kms-sle-removability}. However, there are no known geometric or analytic characterizations that are equivalent to conformal removability \cite{CR_hard}.

  In this work, we will only consider conformally removable curves, whose weldings comprise the space $\CR(\m S^1)$. We endow $\CR(\m S^1)$ with the topology of joint compact convergence for the Riemann maps $f$ and $g$ (i.e., Carath\'eodory topology for both components of $\Chat \setminus \eta$). Let us first make the following observation, which allows us to consider the group of quasisymmetric homeomorphisms acting measurably on the random weldings corresponding to SLE$_\kappa$ loops.
  
    \begin{prop}[See Propositions~\ref{continuous group action} and \ref{continuous group action special point}] \label{Int thm: main thm1 removability}
    The pre- and post-compositions of conformally removable weldings by quasisymmetric homeomorphisms are continuous.
  \end{prop}
    
   Our main result concerns random weldings corresponding SLE$_\k$ loops when $\kappa \in (0,4)$. We focus on a specific realization of the one-to-one correspondence between weldings and loops, described in Definition~\ref{df:SLE welding}. At this time, we note that the SLE$_\k$ loop measure restricted to loops that separate $0$ from $\infty$ induces a probability measure on the stabilizers of $1$ in $\CR(\m S^1)$, which we denote as $\SLE_\kappa^{\mathrm{weld}}$.
\begin{thm}\label{thm:main}
    For $\k\in (0,4)$, sample $\psi_\k$ from $\SLE_\kappa^{\mathrm{weld}}$. We have the following for any fixed $\varphi \in \WP(\m S^1)$ with $\varphi(1)=1$.
\begin{itemize}
    \item The law of $\psi_\k \circ \varphi $ is mutually absolutely continuous with respect to $\SLE_\kappa^{\mathrm{weld}}$.
    \item The law of $\varphi^{-1} \circ \psi_\k $ is mutually absolutely continuous with respect to $\SLE_\kappa^{\mathrm{weld}}$ if the log-ratio $\log\abs{(\varphi(\cdot)-\varphi(1))/(\cdot-1)} $ belongs to $H^{1/2}(\m S^1)$.
\end{itemize}
\end{thm}
We expect that if $\varphi$ is a quasisymmetric circle homeomorphism which is not in $\WP(\m S^1)$, then neither $\psi_\k \circ \varphi$ nor $\varphi^{-1} \circ \psi_\k$ has a law equivalent to $\SLE_\kappa^{\mathrm{weld}}$.

As far as the authors are aware, before this work, only analytic deformations of SLE loops (hence, equivalently, of their welding homeomorphisms) have been studied, within the framework of conformal restriction in particular. This is the first time a non-smooth deformation of SLE is considered. 
As communicated to us by the authors, Baverez and Jego \cite{BJ25} have an independent work studying the SLE welding measure under composition by an analytic circle diffeomorphism, and they compute the Radon--Nikodym derivatives under such deformation using an approach different from but complementary to ours.
We will comment on related literature in Section~\ref{sec:comments}.

\begin{remark}
    For each $\varphi \in \WP(\m S^1)$, the circle homeomorphism $\hat \varphi(\cdot):=\varphi(z_0 \cdot)/\varphi(z_0)$ obtained by conjugating $\varphi$ with the rotation $z\mapsto z_0z$ satisfies $\log\abs{(\hat\varphi(\cdot)-\hat\varphi(1))/(\cdot-1)} \in H^{1/2}(\m S^1)$ for a.e.\ $z_0 \in \m S^1$, as shown in Lemma~\ref{lem log difference}. 
    Instead of normalizing weldings corresponding to $\SLE_\k$ loops by considering those that stabilize $1\in \m S^1$, if we consider the post-composition of $\psi_k$ by a uniform rotation of $\m S^1$, then we obtain a version of Theorem \ref{thm:main} that is symmetric for pre- and post-compositions by elements of $\WP(\m S^1)$: see Corollary~\ref{thm:quasi-invariance of the uniform SLE welding}.
   \end{remark}
\subsection{Key ingredients of the proof}
Our strategy for the proof of Theorem \ref{thm:main} is as follows. We first show that the log-correlated Gaussian field (LGF) on the unit circle is quasi-invariant under pullbacks by Weil--Petersson homeomorphisms. Then, we ``lift'' this characterization to that of the Gaussian multiplicative chaos (GMC). Finally, we obtain our main theorem by identifying the law of SLE welding in terms of compositions of two homomorphisms defined using GMC measures. Here, we explain this proof outline in further detail.

Given an integrable function $h$ on $\m S^1$, we define $\pi_0 h$ to be the projection to its ``mean-zero part.'' That is, $\pi_0 h = h - c_0$ where $c_0$ is the average of $h$ on $\m S^1$.
The LGF on the unit circle is defined as the formal sum
\begin{equation}\label{Int eq:LGF}
    h(\cdot)={\sqrt 2}\sum_{n \geq 1} \xi_n h_n(\cdot), 
\end{equation}
where $\xi_n$ is a sequence of i.i.d.\ standard Gaussian random variables and $h_n$ is an orthonormal basis of $\mc H_0:=\pi_0 H^{1/2}(\m S^1)$. Here, we write $\sqrt{2}$ to demonstrate that the covariance of $h(x)$ and $h(y)$ is $-2\log\abs{x-y}$. 

Let $\varphi \in \Homeo_+(\m S^1)$ be an orientation preserving homeomorphism. Then, $\varphi$ induces a pullback operator $ \Pi(\varphi)$ given by  
\begin{align}\label{Int Eq df of pullback operator}
    \Pi(\varphi)(f):=\pi_0(f\circ \varphi) \quad\text{for}~f \in \mc H_0.
\end{align}
The operator was introduced in \cite{NS95} and \cite{Pa97}; the former showed that $\varphi \in \QS(\m S^1)$ if and only if $\Pi(\varphi)$ is a bounded operator from $\mc H_0$ onto $\mc H_0$. We extend $\Pi(\varphi)$ to act on the LGF $h$ using the expansion \eqref{Int eq:LGF} and by linearity.

\begin{thm}\label{Int prop:quasi-invariance of LGF}
    Let $h$ be an LGF on the unit circle and suppose that $\varphi \in \QS(\m S^1)$. Then, the law of $\Pi(\varphi)(h)$ is mutually absolutely continuous with LGF if and only if $\varphi \in \WP(\m S^1)$. Otherwise, they are mutually singular.
\end{thm}
The proof of Theorem~\ref{Int prop:quasi-invariance of LGF} is based on the Feldman--H\'ajek theorem, which gives a necessary and sufficient condition for classifying two infinite dimensional Gaussian measures on a locally convex space as either mutually absolutely continuous or mutually singular. In our context,  this depends on whether $\Pi(\varphi)\Pi(\varphi)^*-I$ is Hilbert--Schmidt. We show that this condition holds if and only if $\varphi \in \WP(\m S^1)$ (see Lemma~\ref{lem properties of Pi_var}) based on previous results \cite{Sch81,TT06,HS12} which identified certain operators associated with $\Pi(\varphi)$ to be Hilbert--Schmidt.
    
    Our next step is to ``lift'' Theorem \ref{Int prop:quasi-invariance of LGF} to that for the Gaussian multiplicative chaos (GMC) measure $\mc M_h^\g$ defined from LGF $h$ heuristically as $\exp(\frac{\g}{2}h)\,\dd\theta$ for $\g\in (0,2]$. See Section \ref{section: GMC} for the precise definition and its basic properties. We define the normalized GMC measure as $\widehat{\mc M}_h:=\mc M_h/\mc M_h(\m S^1)$ such that the total measure on the circle is equal to $1$.
\begin{prop}\label{Int prop:quasi-invariance of weighted Liouville measure}
    Let $\mc M_h = \mc M_h^{\gamma}$ be the GMC measure corresponding to an LGF $h$ for $\gamma \in (0,2]$. If $\varphi \in \WP(\m S^1)$, then the following coordinate change rule holds for its pull-back: almost surely, 
    \begin{equation}\label{Int eq:coordinate-change}
        \varphi^* \mc M_h  = \mc M_{h\circ \varphi+Q\log\abs{\varphi'}} 
    \end{equation}
where $Q = \frac{\gamma}{2} + \frac{2}{\gamma}$. Furthermore, for $\varphi \in \Homeo_+(\m S^1)$, the law of the normalized pull-back measure $\varphi^* \widehat{\mc M}_h$ is absolutely continuous with respect to the law of $\widehat{\mc M}_h$ if and only if $\varphi \in \WP(\m S^1)$. 
\end{prop}

The coordinate change formula \eqref{Int eq:coordinate-change} was proven to hold for 2D log-correlated field in \cite{DS11_LQG_KPZ}. This proof can be applied straightforwardly for the 1D LGF when $\varphi$ is a diffeomorphism of $\m S^1$ (Lemma~\ref{lem:diff-coord-change}). We prove the first part of Proposition \ref{Int prop:quasi-invariance of weighted Liouville measure} by approximating an arbitrary $\varphi \in \WP(\m S^1)$ with diffeomorphisms and showing that the corresponding GMC measures converge (Lemma~\ref{lem:gmc-convergence}). To show the second part of Proposition \ref{Int prop:quasi-invariance of weighted Liouville measure}, we use that the GMC measure almost surely determines the LGF as proved in \cite{bss-equiv-gmc} (for the $\g=2$ case, in \cite{vihko-equiv-gmc}).

The final step of our proof of Theorem \ref{thm:main} is based on the theory of conformal welding of Liouville quantum gravity (LQG) surfaces. In particular, Ang, Holden, and Sun \cite{AHS23} proved that the conformal welding of two independent LQG disks gives the SLE$_\kappa$ loop measure. We use the following translation of this result in terms of the law of the welding homeomorphism.

 If $h$ is an LGF or a variant thereof, let
\begin{equation*}
    \phi_h^\gamma(z):= \exp(2\pi \ii \cdot \widehat {\mc M}_h^\gamma([1,z])),
\end{equation*}
where $[1,z] \subset \m S^1$ denotes the arc running counterclockwise from 1 to $z$.
\begin{lem}\label{Int lem:welding-homeo}
    Let $\gamma \in (0,2)$ and $\kappa = \gamma^2 \in (0,4)$. If $h_1$ and $h_2$ are independent LGFs on the unit circle, then $\SLE_\kappa^{\mathrm{weld}}$ is mutually absolutely continuous with respect to the law of $(\phi_{h_2 -\gamma\log |\cdot - 1|}^\g)^{-1} \circ \phi_{h_1}^\g$.
\end{lem}
Lemma~\ref{Int lem:welding-homeo} is a short version of Lemma~\ref{lem:welding-homeo}. We will obtain Lemma~\ref{lem:welding-homeo} using conformal welding of quantum disks when each disk has one interior marked point, as established in \cite{ang2024sleloopmeasureliouville} based on the works \cite{Quantum_zipper,AHS23}. Theorem \ref{thm:main} follows from combining Proposition \ref{Int prop:quasi-invariance of weighted Liouville measure} with Lemma \ref{Int lem:welding-homeo}. The conformal welding result for quantum disks in \cite{AHS23} is expected to hold for $\g=2$, in which case Theorem \ref{thm:main} would extend to $\k = 4$.

\subsection{Comments and related literature}\label{sec:comments}

The SLE$_\kappa$ loop measure can be defined for $\kappa \in (0,8)$ \cite{zhan2020sleloop}. When $\kappa \in (4,8)$, SLE$_\k$ is not simple \cite{Rohde_Schramm}. However, it was recently proved in \cite{ang2024sleloopmeasureliouville} that the conformal welding of generalized quantum disks (which have the topology of infinitely many disks concatenated into a tree-like shape) gives the SLE$_\kappa$ loop for $\kappa \in (4,8)$. It would be interesting to consider if there is a family of homeomorphisms on the boundary of a generalized quantum disk that gives an analgous statement to Theorem \ref{thm:main}.

Several recent works have considered conformal deformations of SLE loops. The conformal restriction covariance of chordal SLE was outlined by Lawler, Schramm, and Werner in \cite{LSW_CR_chordal}. Its loop version was postulated by Kontsevich and Suhov \cite{KS07} and proved by Zhan in \cite{zhan2020sleloop}.
The work \cite{SungWang} computed the first variation of the driving function of a Loewner chain under quasiconformal deformations away from the curve, leading to the alternative proof of Loewner energy as a K\"ahler potential on the Weil--Petersson Teichm\"uller space. Based on this result, Gordina, Qian, and Wang \cite{gqw_virasoro_loop} used the SLE$_\k$ loop measure for $\k\in (0,4]$ to construct a natural representation of the Virasoro algebra of central charge $c \leq 1$. An independent work of Baverez and Jego \cite{bj_cft_loop} developed the conformal field theory for the SLE$_\k$ loop measure and proved its characterization as a Malliavin--Kontsevich--Suhov measure for $\k \in (0,4]$. 
On the one hand, these results provide evidence for the idea that the Weil--Petersson Teichm\"uller space should be the Cameron--Martin space of SLE loop measures. On the other hand, deformations considered in these works are limited to those which are analytic in a neighborhood of the Jordan curve. In this work, we consider the full class of Weil--Petersson quasisymmetric homeomorphisms acting on SLE loops for the first time.

Our proof of Theorem \ref{thm:main} can be adapted to show quasi-invariance for other weldings derived from GMC measures that have appeared in the literature. The work \cite{AJKS} showed, using geometric function theory methods without reference to LQG theory, that $\phi_h^\g$ is a welding for $\g\in(0,2)$. The same work also showed that $\phi_{h_2}^{\g_2} \circ (\phi_{h_1}^{\g_1})^{-1}$, where $h_1,h_2$ are independent LGFs and $\g_1,\g_2\in (0,2)$, is a welding. Using a similar approach, the articles \cite{inverse-gmc-welding,CWofGMC} showed that $(\phi_{h_2}^{\g_2})^{-1} \circ \phi_{h_1}^{\g_1}$, which is more alike the random homeomorphism in Lemma \ref{Int lem:welding-homeo}, is a welding for small $\g_1,\g_2>0$. In fact, \ref{Int lem:welding-homeo} shows that the  Jordan curve solving the welding problem for $(\phi_{h_2}^{\g})^{-1} \circ \phi_{h_1}^{\g}$ is locally mutually absolutely continuous with the chordal SLE$_{\gamma^2}$ curve if we remove a neighborhood of the root (the image of $1 \in \m S^1$) of the loop. 
We give the analogous quasi-invariance results for these random weldings in Corollary~\ref{cor quasi-invariance of GMC welding}.

This paper is organized as follows. In Section~\ref{section:QS and WP}, we introduce the groups of circle homeomorphisms that are of interest in this work and describe properties of the associated pull-back operator $\Pi(\varphi)$ on the function space $\mc H_0$. In Section~\ref{section: GMC}, we introduce the LGF and GMC and prove their quasi-invariance under pullbacks by Weil--Petersson homeomorphisms. Section~\ref{section: main result} completes the proof of the quasi-invariance of SLE welding, with the necessary development of the following concepts: conformally removable weldings and the continuity of the composition action on it by the quasisymmetric group, the SLE loop measure and the corresponding SLE welding measure, and the conformal welding of quantum disks.    

\vspace{8pt}

\textbf{Acknowledgements.} The authors wish to thank Yilin Wang for her invaluable insight into the relationship between SLE and the universal Teichm\"uller space. 
We are grateful to Guillaume Baverez and Antoine Jego for sharing their independent manuscript \cite{BJ25}.
We are also grateful to Hong-Bin Chen, Eero Saksman, Fredrik Viklund, and Hao Wu for helpful discussions.

This work was completed in part during J.S.'s visit to the IHES and S.F.'s visit to the IMSI at the University of Chicago, and the authors wish to thank the respective institutions for their hospitality. S.F. is funded by Beijing Natural Science Foundation (JQ20001); the European Union (ERC, RaConTeich, 101116694) and Tsinghua scholarship for overseas graduates studies (2023076). J.S. is partially supported by a fellowship from Kwanjeong Educational Foundation.

\section{Pullback operator associated with the Weil--Petersson class}\label{section:QS and WP}
In this section, we introduce the group of quasisymmetric circle homeomorphisms and the associated pullback operators on a function space of the unit circle and the log-ratio. The first result of this section is Lemma \ref{lem properties of Pi_var}, where we find various correspondences between subgroups of circle homeomorphisms and the properties of pullback operators. The second result is Lemma~\ref{lem log difference}, where we prove that the log-ratio is in the same space as the log-derivative.

\subsection{Quasisymmetric homeomorphisms and the Weil--Petersson class }
Let $\m S^1$ denote the unit circle, which we identify with the boundary of the unit disk $\m D$ in the complex plane. The set of orientation-preserving homeomorphisms of $\m S^1$, which we denote $\Homeo_+(\m S^1)$, has a natural group structure with the composition of functions as the group action.
In this subsection, we introduce various subgroups of it. Some trivial examples are $\Diff_+(\m S^1)$, $\mob(\m S^1)$, and $\mathrm{Rot}(\m S^1)$, which consist of smooth diffeomorphisms, M\"obius transformations, and rotations, respectively. 

We shall pay special attention to the subgroup of quasisymmetric homeomorphisms. We briefly recall its definition and basic properties for the reader who is unfamiliar with the concept and direct to, e.g., \cite{lehto2012univalent} for further details.

\begin{df}\label{df:QS}
    We say that $\varphi \in \Homeo_+(\m S^1)$ is \textit{quasisymmetric} if there exists some $C_0 > 0$ such that 
    \begin{equation}
        \frac{1}{C_0} \leq \left|\frac{\varphi(e^{\ii (\theta + t )})-\varphi(e^{\ii \theta })}{\varphi(e^{\ii \theta })-\varphi(e^{\ii (\theta - t )})}\right| \leq C_0 
    \end{equation}
    for all $\theta \in \m R$ and $t \in (0,2\pi)$. Let $\QS(\m S ^1)$ denote the group of orientation preserving quasisymmetric homeomorphisms of the unit circle $\m S^1$. 
\end{df}

    Beurling and Ahlfors \cite{BA56} proved that $\varphi\in \Homeo_+(\m S^1)$ is quasisymmetric if and only if there exists some quasiconformal homeomorphism $\omega$ of $\m D$ onto itself that extends continuously to $\varphi$ on $\m S^1 = \partial \m D$. That is, the \textit{Beltrami coefficient} 
    \begin{equation*}
        \mu_\o = {\partial_{\bar{z}} \omega}/{\partial_z \omega }
    \end{equation*}
    of $\o$ is defined almost everywhere on $\m D$ and satisfies $\norm{\mu_\o}_{L^\infty(\m D)}:=\sup_{z \in \m D}\abs{\mu_\o(z)}<1$. Heuristically speaking, $\omega$ maps small circles centered at $z\in \m D$ to ellipses with eccentricity $K_\o(z)$, where the function
    \begin{equation*}
        K_\o=\frac{1+\abs{\mu_\o}}{1-\abs{\mu_\o}},
    \end{equation*}
    is called the \textit{dilatation} of $\o$. The Teichm\"uller distance  between two quasisymmetric homeomorphisms is defined as
    \begin{equation*}
        \tau_1(\varphi_1,\varphi_2)=\inf\left\{\frac{1}{2}\log\frac{1+\norm{\frac{\mu_1-\mu_2}{1-\bar\mu_1\mu_2}}_{L^\infty(\m D)}}{1-\norm{\frac{\mu_1-\mu_2}{1-\bar\mu_1\mu_2}}_{L^\infty(\m D)}}\middle|\varphi_1=\o_{\mu_1}|_{\m S^1},\varphi_2=\o_{\mu_2}|_{\m S^1}\right\}.
    \end{equation*}
    The Teichm\"uller distance induces the natural topology on $\QS(\m S^1)$. Let us further introduce two special subgroups of $\QS(\m S^1)$.
    \begin{df}\label{df:Sy}
        We say that $\varphi\in \QS(\m S^1)$ is \textit{symmetric} if $\varphi$ can be extended to a quasiconformal map $\omega$ on $\m D$ whose Beltrami coefficient $\mu_\o$ satisfies $\mu_\o(z)\rightarrow 0$ as $\abs{z}\rightarrow1$. Let $\Sy(\m S^1)$ denote the sets of all symmetric orientation preserving homeomorphisms of $\m S^1$.     \end{df}

        The following subgroup, first introduced in \cite{cui00}, is our protagonist. 
 \begin{df}\label{df:WP}
     We say that $\varphi\in \QS(\m S^1)$ belongs to the\textit{ Weil--Petersson class} if $\varphi$ has a quasiconformal extension $\omega$ to the unit disk whose Beltrami coefficient $\mu_\o$ satisfies
    \begin{equation}\label{eq:WP norm}
        \int_\m D \abs{\mu_\o(z)}^2 (1-\abs z^2)^{-2} \,\dd A(z) < \infty,
    \end{equation}
    where $A$ denotes the area measure.
    Equivalently, $\varphi$ is absolutely continuous (with respect to the arc-length measure) and $\log |\varphi'| \in H^{1/2}(\m S^1)$. Here, $H^{1/2}(\m S^1)$ is the fractional Sobolev space consisting of functions $f:\m S^1 \to \m R$ satisfying 
        \begin{equation} \label{eq 1/2 seminorm}
            \iint_{\m S^1 \times \m S^1} \left|\frac{f(x)-f(y)}{x-y}\right|^2 \dd x \,\dd y < \infty.
        \end{equation}
    Let $\WP(\m S^1)$ denote the set of all quasisymmetric homeomorphisms of the unit circle that belong to the Weil--Petersson class.
 \end{df}
 The equivalence between the two definitions above is due to Yuliang Shen \cite{Shen18}. Moreover, it is proved there that the above two metric induces the same topology. The following is an equivalent definition of $H^{1/2}(\m S^1)$.
 \begin{equation*}
       H^{1/2}(\m S^1)=\left\{f: \m S^1 \rightarrow \m R ~\bigg\vert~ f(e^{\ii \theta})=c_0+\sum_{n\in \m Z\setminus \{0\}} c_n \frac{e^{\ii n\theta}}{\sqrt{|n|}},~\sum_{n=1}^\infty |c_n|^2 < \infty \right\}.
   \end{equation*} We discuss this space in further detail in the next subsection. There is a further multitude of equivalent definitions for the Weil--Petersson class related to various parts of mathematics: see \cite{Bishop_WP} for a compilation. 

 Here is the relationship between the groups of circle homeomorphisms we have considered in this section so far.
    \begin{equation*}
        \Diff_+(\m S^1) \subsetneqq \WP(\m S^1)\subsetneqq \Sy(\m S^1)\subsetneqq \QS(\m S^1) \subsetneqq \CR(\m S^1) \subsetneqq \Homeo_+(\m S^1).
    \end{equation*}

  \begin{remark}
      The universal Teichm\"uller space $T(1)$ and the Weil--Petersson Teichm\"uller space $T_0(1)$ can be represented by $\mob(\m S^1)\backslash \QS(\m S^1)$ and $\mob(\m S^1)\backslash \WP(\m S^1)$, respectively. Identifying these coset spaces with the subgroup of homeomorphisms that fix $-1, -\ii$, and 1, the group structure on these spaces is given by composition. As hinted in the introduction, taking the quotient under left actions by $\mob(\m S^1)$ corresponds to considering the equivalence class of quasicircles on $\hat{\m C}$ modulo M\"obius transformations that fix $0$. The norm \eqref{eq:WP norm} is inherited from the Weil--Petersson metric, which gives the universal Teichm\"uller space a Hilbert structure \cite{TT06}.

      In our work, we shall be mostly interested in the Weil--Petersson Teichm\"uller curve $\mc T_0(1) \simeq \mathrm{Rot}(\m S^1)\backslash \WP(\m S^1)$, which we identify with the cosets of homeomorphisms that fix 1. By conformal welding, these homeomorphisms can be identified with quasicircles in $\m C$ that disconnect 0 from $\infty$ modulo M\"obius transformations that fix 0 and $\infty$. See \cite[Sec.\ I.1]{TT06} for further details.
  \end{remark}
 
\subsection{Sobolev spaces on the unit disk and its boundary}
In this subsection, we firstly introduce the symplectic Hilbert space $\mc H_0$ consisting of elements of the fractional Sobolev space $H^{1/2}(\m S^1)$ with zero mean. Looking ahead, the significance of the space $\mc H_0$ is that it is the Cameron--Martin space of the the log-correlated Gaussian field on $\m S^1$ (see Section \ref{section: GMC}).
Then we consider its compexification and the Poisson integral. Finally, we mention other Sobolev spaces.

Let $\mc H_0$ denote the real Hilbert space
\begin{equation}
\begin{split}
    \mc H_0 :=& \left\{f: \m S^1 \rightarrow \m R ~\bigg\vert~ f(e^{\ii \theta})=\sum_{n\neq 0} c_n \frac{e^{\ii n\theta}}{\sqrt{|n|}} \text{ with } c_{-n} = \overline{c_n},~\sum_{n=1}^\infty |c_n|^2 < \infty \right\}
    \end{split}
\end{equation}
with the inner product 
\begin{equation} \label{eq inner product}
    \left\langle \sum_{n\neq 0} c_n \frac{e^{\ii n\theta}}{\sqrt{|n|}},\sum_{n\neq 0} d_n \frac{e^{\ii n\theta}}{\sqrt{|n|}} \right\rangle = \sum_{n\neq 0} c_n \overline {d_n}.
\end{equation}

We consider the canonical symplectic form $\T$ on $\mc H_0$ introduced in \cite{NS95} as
\begin{equation}\label{equation symplectic form}
    \T(f,g)=\frac{1}{2\pi}\int_{\m S^1} f ~\mathrm{d}g = -\ii \sum_{n=1}^\infty (c_n \bar d_n - c_{-n} \bar d_{-n}).
\end{equation}
for $f(e^{\ii \theta}) = \sum_{n\neq 0} c_n \frac{e^{\ii n\theta}}{\sqrt{|n|}}$ and $g(e^{\ii \theta}) = \sum_{n\neq 0} d_n \frac{e^{\ii n\theta}}{\sqrt{|n|}}$.
 Let us denote the group of bounded symplectomorphisms of $\mc H_0$ as $\operatorname{Sp}(\mc H_0)$.

By complex linearity, the symplectic form $\T$ defined in \eqref{equation symplectic form} extends to the complexification 
\begin{equation}
    \mc H_0^\m C := \left\{f: \m S^1 \rightarrow \m C ~\bigg\vert~ f(e^{\ii \theta})=\sum_{n\neq 0} c_n \frac{e^{\ii n\theta}}{\sqrt{|n|}},~\sum_{n\neq 0} |c_n|^2 < \infty \right\}.
\end{equation}
With respect to $\T$, the Hilbert space $\mc H_0^\m C$ has a canonical decomposition into two closed isotropic subspaces
\begin{equation}
    \mc H_0^\m C = W_+ \oplus W_-,
\end{equation}
where 
\begin{align}
    W_+ &=\left\{f: \m S^1 \rightarrow \m C ~\bigg\vert~ f(e^{\ii \theta})=\sum_{n = 1}^\infty  a_n \frac{e^{\ii n\theta}}{\sqrt{n}}, \,\sum_{n= 1}^\infty |a_n|^2 < \infty \right\},\\
    W_- &=\left\{f: \m S^1 \rightarrow \m C ~\bigg\vert~ f(e^{\ii \theta})=\sum_{n = 1}^\infty  b_n \frac{e^{-\ii n\theta}}{\sqrt{n}}, \,\sum_{n= 1}^\infty |b_n|^2 < \infty \right\}. 
\end{align}
Let
\begin{align}\label{eq:basis}
    \left\{e_n=\frac{e^{\ii n\theta}}{\sqrt{n}}\right\}_{n \geq 1} ~\text{and}~ \left\{f_n=\frac{e^{-in\theta}}{\sqrt{n}}\right\}_{n \geq 1}
\end{align}
denote the standard bases of the subspaces $W_+$ and $W_-$, respectively. Under these bases, $W_+$ and $W_-$ are naturally isomorphic to $\ell^2(\m C)$.

Each element of $\operatorname{Sp}(\mc H_0)$, the group of bounded symplectomorphism of $\mc H_0$, extends to $\mc H_0^\m C$ again by complex linearity. In the basis $\{e_n\}_{n \geq 1}$ and $\{f_n\}_{n \geq 1}$, they can be represented by the matrices
\begin{align}\label{eq pullback matrix rep}
    \left(\begin{array}{cc}
      M &  N\\
         \Bar{N}& \Bar{M}
    \end{array}\right)
    ~\text{where}~ MM^*-NN^* = I, MN^t=NM^t. 
\end{align}
Above, $M^*$ denotes the adjoint matrix of $M$ and $M^t$ denotes the transpose matrix of $M$.

Let us now consider the relationship between $\mc H_0$ and the Dirichlet class of functions on the unit disk $\m D$.

\begin{itemize}
\item The space $\mc H_0$ is naturally isomorphic to the real Hilbert space $\mc D_0$ of harmonic functions $F$ on the unit disk $\m D$ with $F(0)=0$ and finite Dirichlet energy. That is,
\begin{align}
     \mc D_0 &=\left\{F: \m D \rightarrow \m R ~\bigg\vert~ F(z)=\sum_{n> 0} c_n \frac{z^n}{\sqrt{ n}}+c_{-n}\frac{\bar z^n}{\sqrt{ n}},c_{-n} = \overline{c_n}, \sum_{n=1}^\infty |c_n|^2 < \infty\right\}.
\end{align}
If $F \in \mc D_0$, then it is straightforward to check that the trace $f:= F|_{\m S^1}$ on $\m S^1$ is an element of $\mc H_0$ with 
\begin{equation}\label{eq:norm-compare}
        \|f\|_{\mc H_0}^2 = \frac{1}{2\pi}\int_{\m D} |\nabla F|^2 < \infty .
    \end{equation}
     On the other hand, let $\Po (f)$ denote the Poisson integral of an integrable function $f$ on the unit circle $\m S^1$: i.e.,
\begin{equation}
    \Po (f)(z)=\frac{1}{2\pi \ii}\int_{\m S^1} \Re \frac{w+z}{w-z} \frac{f(w)}{w} ~\dd w,~\text{for}~ z \in \m D.
\end{equation}
Then, for each $f\in \mc H_0$, we have $\Po(f) \in \mc D_0$ with $f = \Po(f)|_{\m S^1}$. 
Clearly, this isomorphism between $\mc H_0$ and $\mc D_0$ given by the trace operator and the Poisson integral extends naturally to that between their complexifications.

\item Let $H^1(\m D)$ (resp.\ $H_0^1(\m D)$) be the real Hilbert space given by the completion of the space of smooth functions on $\m D$ (resp.\ with compact support) with respect to the Dirichlet inner product. Then, we have the decomposition 
\begin{align}\label{eq Hilbert decomposition}
    H^1(\m D)= H_0^1(\m D) \oplus \mc D_0 \oplus \m R
\end{align}
as a direct sum with respect to the Dirichlet inner product. We will see later that, from the decomposition above, there exists a decomposition of the Neumann Gaussian free field into the sum of independent Dirichlet Gaussian free field and the ``harmonic'' Gaussian field on $\m D$ (see Remark~\ref{rem:GFF decomposition}).
\end{itemize}

We conclude this subsection by giving a quick overview of fractional Sobolev spaces of general index $s \in \m R$ on the unit disk $\m D$ and its boundary $\m S^1$.

First, let us consider the space $H_0^s(\m D)$ with zero boundary conditions. 
Let $\{g_n\}_{n \geq 1}$ be a sequence of eigenfunctions of the Laplacian $-\Delta$ on $\m D$ with Dirichlet boundary conditions which are mutually orthogonal with respect to the Dirichlet inner product.  Let $\{\lambda_n\}_{n\geq 1}$ be the corresponding eigenvalues. That is, $g_n$ and $\lambda_n$ satisfy
\begin{align}
    \begin{cases}
           -\Delta g_n = \lambda_n g_n \quad &\text{in}~ \m D\\
           g_n = 0 \quad&\text{on}~ \partial \m D
    \end{cases}
\end{align}
for each $n$. The eigenvalues $\{\lambda_n\}_{n \geq 1}$ are positive and satisfy $\lambda_n \rightarrow \infty $ as $n \rightarrow \infty $.

For each $s \in \m R$, we define $H_0^s(\m D)$ as the real Hilbert space obtained by taking the completion of the set of smooth, compactly supported real-valued functions on the unit disk with respect to the inner product
\begin{equation}\label{eq fractional inner product}
    \langle f, g\rangle _s = \sum_{n \geq 1} \lambda_n^s \langle f, g_n\rangle _{L^2} {\langle g, g_n \rangle _{L^2}} .
\end{equation}
When $s=1$, this inner product is the standard Dirichlet inner product on $\m D$. 

We define the Sobolev space $H^s(\m D)$ with free boundary conditions analogously using the eigenfunctions of $-\Delta$ on $\m D$ with Neumann boundary conditions. For the eigenvalue 0 corresponding to the constant function, we let $0^s:=1$ for all $s\in \m R$ in \eqref{eq fractional inner product}.

Similarly, we can define the Sobolev space $H^s(\m S^1)$ of functions on $\m S^1$ replacing the operator $-\Delta$ by $-\partial_{\theta\theta}$.
Then, $\mc H_0$ agrees with the closed subspace $H^{1/2}(\m S^1)/\m R$ of functions $f\in \mc H$ with $\int_{\m S^1} f = 0$ (``mean zero''). In other words, $\mc H$ consists of functions in $\mc H_0$ plus a constant. Let us denote the natural projection $H^{1/2}(\m S^1) \to \mc H_0$ as
\begin{equation}
    \pi_0(f) := f - \frac{1}{2\pi}\int_{\m S^1} f.
\end{equation}
The equivalence between this definition of $H^{1/2}(\m S^1)$ and the condition \eqref{eq 1/2 seminorm} can be found in introductory texts on fractional Sobolev spaces.
\subsection{The pullback operator}
In this subsection, we introduce the pullback operator as a right group action of $\Homeo_+(\m S^1)$ on $\mc H_0$. Then we prove a key relationship (Lemma \ref{lem properties of Pi_var}) with the Weil--Petersson class. 

\begin{df}

Given a orientation preserving homeomorphism $\varphi \in \Homeo_+(\m S^1)$, we define the pullback operator $ \Pi(\varphi)$ on $\mc H_0^{\m C}$ as
\begin{equation}\label{Eq df of pullback operator}
    \Pi(\varphi)(f):= \pi_0(f\circ \varphi) = f \circ \varphi - \frac{1}{2\pi}\int_{\m S^1} f \circ \varphi ~\mathrm{d}\theta.
\end{equation}
\end{df}
Let us denote the space of bounded linear operators from $\mc H_0^\m C$ to itself as $\ms B(\mc H_0^\m C)$.
Nag and Sullivan \cite{NS95} proved that $\Pi(\varphi)\in \ms B(\mc H_0^\m C)$ if and only if $\varphi \in \QS(\m S^1)$, and that the assignment 
    $\Pi : \QS(\m S^1) \rightarrow \ms B(\mc H_0^\m C)$

defines a right group action on $\mc H_0^\m C$ by symplectomorphisms. In the basis $\{e_n\}_{n \geq 1}$ and $\{f_n\}_{n \geq 1}$ given in \eqref{eq:basis}, the symplectomorphism $\Pi(\varphi)$ for $\varphi \in \QS(\m S^1)$ can be represented by the matrix of the form \eqref{eq pullback matrix rep}, whose entries are given by
\begin{align}\label{eq:op-entry-m}
    M_{mn}(\varphi) &= \frac{1}{2\pi} \sqrt{\frac{m}{n}} \int_{\m S^1} \left(\varphi(e^{\ii \theta})\right)^n e^{-\ii m\theta} \mathrm{d}\theta,\\
\label{eq:op-entry-n}
    N_{mn}(\varphi) &= \frac{1}{2\pi} \sqrt{\frac{m}{n}} \int_{\m S^1} \left(\varphi(e^{\ii \theta})\right)^{-n} e^{-\ii m\theta} \mathrm{d}\theta.
\end{align}
    We note that the operator $\Pi(\varphi)$ preserves the subspaces $W_+$ and $W_-$ (i.e., $\Pi(\varphi)$ belongs to the unitary subgroup $\Un(\mc H_0)$ of $\SP (\mc H_0)$ consisting of bounded symplectomorphisms with $N=0$) if and only if $\varphi \in \mob(\m S^1)$. 

We now give the key result of this section, which will be used in Section \ref{section: GMC} to identify the quasi-invariance of the log-correlated Gaussian field on $\m S^1$ under pullbacks by $\Homeo_+(\m S^1)$. First, let us recall the definition of a Hilbert--Schmidt operator.

\begin{df}
    For any bounded linear operator $T$ from a Hilbert space $H$ to itself, we define the \textit{Hilbert--Schmidt norm} $\norm T_{\HS}$ by 
     \begin{equation}\label{df Hilbert--Schmidt norm}
        \norm T_{\HS}^2 := \sum_{\mr a\in \mc A} \norm {T e_\mr a}_{H}^2 = \TR (T^*T),
     \end{equation}
     where $\{e_\mr a, \mr a \in \mc A \}$ is any orthonormal basis of $H$ and $T^*$ is the adjoint operator of $T$. For a self-adjoint bounded linear operator $A$, the trace $\TR(A)$ of $A$ is defined by the sum of all its eigenvalues if the series is absolutely summable and set to be $+\infty$ otherwise. We say $T$ is \textit{Hilbert--Schmidt} if $\|T\|_{\HS}<\infty$. The collection of all Hilbert--Schmidt operators on $H$ forms a Hilbert space $\HS(H)$ with respect to the norm \eqref{df Hilbert--Schmidt norm}.
\end{df}
Note that if $H^\m C$ is the complexification of $H$, then the complex extension of $T$ belongs to $\HS(H^\m C)$ if and only if $T \in \HS(H)$. If $T \in \HS(H)$ and $S \in \ms B(H)$, then $T^t$, $T^*$, $ST$, and $TS$ belong to $\HS(H)$.
\begin{lem}\label{lem properties of Pi_var}
Assume $\varphi \in \QS(\m S^1)$ and let $\Pi(\varphi)\in \ms B(\mc H_0)$ be the pullback operator defined by \eqref{Eq df of pullback operator}. Then, $\Pi(\varphi)\Pi(\varphi)^*-I$ is Hilbert--Schmidt if and only if $\varphi \in \WP(\m S^1)$.
\end{lem}
\begin{proof}
Recall the matrices $M$ and $N$ associated with $\Pi(\varphi)$ as given by \eqref{eq:op-entry-m}--\eqref{eq:op-entry-n}. Since $\Pi(\varphi) \in \SP(\mc H_0)$, we have 
\begin{equation*}
    \begin{pmatrix} M & N \\ \Bar{N} & \Bar{M} \end{pmatrix} 
    \begin{pmatrix} M^* & -N^t \\ -N^* & M^t   \end{pmatrix} = I.
\end{equation*}
as in \eqref{eq pullback matrix rep}. Thus,
\begin{equation}\label{eq hs calculation}
    \Pi(\varphi) \Pi(\varphi)^* - I = 
    \begin{pmatrix} M & N \\ \Bar{N} & \Bar{M} \end{pmatrix} 
    \begin{pmatrix} M^* & N^t \\ N^* & M^t   \end{pmatrix} - I = 2\Pi(\varphi) \begin{pmatrix} & N^t \\ N^* &  \end{pmatrix}.
\end{equation}
In \cite[Thm.\ 2.2]{HS12}, it was shown (up to the isomorphism described in the previous subsection) that $N$ is Hilbert--Schmidt if and only if $\varphi \in \WP(\m S^1)$. Combined with the fact that $\Pi(\varphi)$ and $\Pi(\varphi)^{-1} = \Pi(\varphi^{-1})$ are bounded for $\varphi \in \QS(\m S^1)$, we obtain the desired result.
\end{proof}
\begin{remark}
    For the same reason, the operator $\Pi(\varphi)\Pi(\varphi)^*-I$ is compact if and only if $\varphi \in \Sy(\m S^1)$.
\end{remark}
In fact, we will need that all characterizations of $\WP(\m S^1)$ that we have discussed so far induce the same topology, which we check by retracing the proofs of equivalences among the various definitions.

\begin{lem}\label{lem:wp-approx}
    If $\varphi \in \WP(\m S^1)$, then there exists a sequence $\varphi_n \in \Diff_+(\m S^1)$ such that $\sup_{x\in \m S^1}|\varphi_n(x) - \varphi(x)|$, $\norm{\Pi(\varphi_n\circ\varphi^{-1})\Pi(\varphi_n\circ\varphi^{-1})^*-I}_{\HS}$, $\norm{\log|(\varphi_n \circ\varphi^{-1})'|}_{H^{1/2}(\m S^1)}$ and $\norm{\Pi(\varphi_n)u-\Pi(\varphi)u}_{H^{1/2}(\m S^1)}$ for any $u \in H^{1/2}(\m S^1)$
    all converge to 0 as $n\to\infty$.
    \end{lem}
\begin{proof}
    It is shown in \cite[Thm.~1.4]{Shen18} that the $H^{1/2}(\m S^1)$ norm for $\log\abs{\varphi'}$ induces the same topology as the Weil--Petersson metric. It is shown in \cite[Prop.~4.1]{Shen18} that the topology for $\varphi$ induced by the Teichm\"uller distance is finer than the strong operator topology for $\Pi(\varphi)$. From the proof of \cite[Thm.~4.2]{HS12}, when the dilatation is bounded, this is equivalent to the Hilbert--Schmidt norm for the matrix $N$ with entries \eqref{eq:op-entry-n} and therefore the Hilbert--Schmidt norm for the matrix $\Pi(\varphi)\Pi(\varphi)^*-I$ by \eqref{eq hs calculation}. The uniform convergence follows from the argument using the normal family. More precisely, up to subsequence, we can choose $\varphi_n$ to be the boundary of the quasiconformal self-homeomorphism $\o_n$ of $\m D$  that agrees with $\varphi$ at $1,\ii,-1$ with Beltrami coefficient 
        \begin{equation*}
            \mu_{n}(z) = \mu(z) \mathbf{1}_{\{|z| < 1-1/n\}}
        \end{equation*}
        where $\mu$ is the Beltrami coefficient of the Douady--Earle extension $\o$ of $\varphi$ to $\m D$.
\end{proof}
\subsection{Estimates on the log-ratio}
Now, we aim to informally replace the derivative with the difference quotient, but we actually use conformal welding. Looking ahead, it corresponds to adding one boundary marked point on the quantum disk, see Corollary~\ref{cor: one marked boundary}.
Let $H^{1/2}(\m S^1,\m C)$ denote the complexification of $H^{1/2}(\m S^1)$.
\begin{lem}\label{lem log difference}
    If $\varphi \in \WP(\m S^1)$, then for a.e. $ z_0 \in \m S^1$,
    \begin{equation}\label{eq:diff ratio}
        u_\varphi(\cdot,z_0):=\log{\frac{\varphi(\cdot)-\varphi(z_0)}{\cdot-z_0}} \in H^{1/2}(\m S^1,\m C).
    \end{equation}
    Moreover, the function $z_0 \mapsto \norm{u_\varphi(\cdot,z_0)}_{H^{1/2}(\m S^1, \m C)}$ is $L^2$-integrable.
\end{lem}
\begin{proof}
    Recall from the introduction the conformal welding decomposition for $\varphi \in \WP(\m S^1)$: there exist quasiconformal maps $f$ and $g$ on $\hat{\m C}$, conformal on $\m D$ and $\m D^*$, respectively, such that $\varphi=(g^{-1} \circ f)|_{\m S^1}$, which is also a direct corollary of Lemma~\ref{lem:removability is QS-invariant under composition} when $\psi$ is the identity map. Since \begin{equation*}
        u_\varphi(\cdot,z_0)=\log{\frac{f(\cdot)-f(z_0)}{\cdot-z_0}} - \log{\frac{g\circ \varphi(\cdot)-g\circ \varphi(z_0)}{\varphi(\cdot)-\varphi(z_0)}}   =: u_f (\cdot,z_0)-u_g(\varphi(\cdot),\varphi(z_0)) 
    \end{equation*}
    and $H^{1/2}(\m S^1,\m C)$ is invariant under pullback by a quasisymmetric map, we only need to show that $u_f = u_f(\cdot,z_0)$ and $u_g =u_g(\cdot,\varphi(z_0))$ belong to $H^{1/2}(\m S^1,\m C)$ for almost every $z_0 \in  \m S^1$.

    It is shown in \cite[Chap~2, Lemma~2.5]{TT06} that $u_f(z,w)=\sum_{n,m=0}^\infty c_{n,m} {z^{n} w^{m}}$ with
    \begin{equation*}
        \sum_{n,m=1}^\infty \abs{nm}|c_{n,m}|^2, \sum_{n=1}^\infty \abs{n}|c_{n,0}|^2, \sum_{m=1}^\infty \abs{m}|c_{0,m}|^2< \infty,
    \end{equation*}
    which is equivalent to the associated Grunsky operator being Hilbert--Schmidt.
     Define $F_n(e^{\ii \psi}):=\sum_{m=0}^{\infty} c_{n,m}e^{\ii m \psi}$ such that $\norm{F_n}_{H^{1/2}(\m S^1,\m C)}^2=\sum_{m=1}^\infty \abs{m}|c_{n,m}|^2$, then $$\norm{u_f(\cdot,z_0)}_{H^{1/2}(\m S^1,\m C)}^2=\sum_{n=1}^{\infty}\abs{n}\abs{F_n(z_0)}^2=:F(z_0).$$ It follows that
     \begin{equation*}
         \norm{F}_{L^1(\m S^1)}= \sum_{n=1}^{\infty}\abs{n}\norm{F_n}_{L^2(\m S^1)}^2\leq \sum_{n=1}^{\infty}\abs{n}(\norm{F_n}_{H^{1/2}(\m S^1,\m C)}^2+|c_{n,0}|^2)<\infty.
     \end{equation*}
     The case for $g$ is similar.
\end{proof}

\begin{remark}
    When $\varphi \in C^3(\m S^1)\supsetneqq \WP(\m S^1)$ is three times continuously differentiable, the proof is straightforward at every point. 
\end{remark}
\begin{remark}\label{rem: equivalent definition of WP}
    For the same reason, plus controlling the term $c_{n,m}$ when $nm=0$, we can show that if $\varphi \in \WP(\m S^1)$, then
    \begin{align}\label{eq two variant}
    &u_\varphi(e^{\ii \t},e^{\ii \psi})=\sum_{n,m=-\infty}^\infty c_{n,m} {e^{\ii (n\theta+m\psi)}},\\ \label{eq norm two variant}
        &\text{where}~\sum_{n,m=-\infty}^\infty \abs{nm}|c_{n,m}|^2,\sum_{n=-\infty}^\infty \abs{n}\abs{c_{n,0}} ^2,\sum_{m=-\infty}^\infty \abs{m}\abs{c_{0,m}}^2 
    < \infty.
    \end{align}
    It is nothing but the condition for the covariance functions of two equivalent LGFs. So we conjecture that \eqref{eq two variant} and \eqref{eq norm two variant} hold if and only if $\varphi \in \WP(\m S^1)$ without assuming $\varphi \in \QS(\m S^1)$.
\end{remark}

\section{Quasi-invariance of Gaussian fields and boundary measures}\label{section: GMC}

In this section, we show that the law of the log-correlated Gaussian field on $\m S^1$ is quasi-invariant under the pullback by $\varphi \in \QS(\m S^1)$ if and only if $\varphi$ is of the Weil--Petersson class. We then ``lift'' this characterization to that for the Gaussian multiplicative chaos on $\m S^1$.
\subsection{Preliminaries of Log-correlated Gaussian field on the unit circle}

In this subsection, we survey the definition and the basic properties of the log-correlated Gaussian field on the unit circle. This is a well-researched object, and we do not aim to be comprehensive in our introduction; we direct the reader to surveys such as \cite{LGF-survey} for further information.

The \textit{(mean-zero) log-correlated Gaussian field} on $\m S^1$ (denoted \textbf{LGF}) is a random generalized function on $\m S^1$ which has a centered Gaussian law with Cameron--Martin space $\mc H_0$ and the Cameron--Martin norm $\|\cdot\|_{\mc H_0}$ That is, 
\begin{equation}\label{eq lgf def}
    h = \sum_{n\geq 1} \xi_n h_n,
\end{equation}
where $\{h_n\}_{n\geq 1}$ is an orthonormal basis of $\mc H_0$ comprised of continuous functions and $\{\xi_n\}_{n\geq 1}$ is an i.i.d.\ sequence of standard normal random variables, is an instance of LGF on $\m S^1$. For concreteness, we can take 
\begin{equation}\label{eq trigonometric basis}
\begin{split}
    h_{2n-1}(e^{\ii \theta}) &:= \frac{e_n + f_n}{\sqrt 2}(e^{\ii \theta}) = \sqrt{\frac{2}{n}} \cos (n\theta),\quad
    h_{2n}(e^{\ii \theta}) := \frac{e_n - f_n}{\sqrt 2 \ii }(e^{\ii \theta}) = \sqrt{\frac{2}{n}} \sin (n\theta),
\end{split}\end{equation}
where $\{e_n,f_n\}_{n\geq 1}$ is the standard basis for $\mc H_0^{\m C}$ introduced in \eqref{eq:basis}. However, the definition \eqref{eq lgf def} is independent of the choice of the orthonormal basis $\{h_n\}_{n\geq 1}$ for $\mc H_0$. 

The sum \eqref{eq lgf def} can be seen to converge almost surely in $H^s(\m S^1)$ for any $s<0$. Equivalently, we may consider $h$ in terms of the centered Gaussian process $\{\langle h,f\rangle\}_{f\in \mc H_0}$ where $\langle \, , \,\rangle$ is the inner product \eqref{eq inner product} on $\mc H_0$. That is, 
\begin{equation}
    \left\langle h, \sum_{n\geq 1} c_n h_n\right\rangle := \sum_{n\geq 1} \xi_n c_n
\end{equation}
for $\sum_{n\geq 1} c_n h_n \in \mc H_0$. Note that for $f\in \mc H_0$, we have $\var(\langle h,f\rangle) = \|f\|_{\mc H_0}^2$.
More generally, we consider
\begin{equation}\label{eq lgf process def}
    (h,\rho) := \sum_{n\geq 1} \xi_n \int_{\m S^1} h_n(e^{\ii \theta})\, \dd \rho(e^{\ii \theta})
\end{equation}
with signed Borel measures $\rho$ on $\m S^1$ for which the above sum is almost surely absolutely convergent. Then, such pairings $(h,\rho)$ form a continuous version of the centered Gaussian process satisfying
\begin{equation}\label{eq lgf cov formula}
    \COV((h,\rho),(h,\tilde \rho)) = \iint_{\m S^1\times \m S^1} G(e^{\ii \theta}, e^{\ii \tilde \theta})\,\dd \rho(e^{\ii \theta})\, \dd \tilde \rho(e^{\ii \tilde \theta}),
\end{equation}
where 
\begin{equation}\label{eq lgf cov kernel}
    G(e^{\ii \theta},e^{\ii \tilde \theta}) := \sum_{n\geq 1} h_n(e^{\ii\theta}) h_n(e^{\ii \tilde \theta}) = -2\log \big|e^{\ii \theta} - e^{\ii \tilde \theta}\big|
\end{equation}
is the \textit{covariance kernel} of the LGF $h$. Note that for $f\in \mc H_0$, we have 
\begin{equation}
    \langle h,f \rangle = \left(h, \frac{(-\partial_{\theta\theta})^{1/2} f(e^{\ii\theta})}{2\pi}\,\dd\theta\right),
\end{equation}
where $(-\partial_{\theta\theta})^{1/2}$ is the linear operator from $\mc H_0$ to $L^2(\m S^1)$ which maps $\cos(n\theta)$ to $n\cos(n\theta)$ and $\sin(n\theta)$ to $n \sin (n\theta)$ for each positive integer $n$.

This definition of the LGF on $\m S^1$ is similar to that of the \textit{Gaussian free field} (GFF) on the unit disk $\m D$, which is given by the sum \eqref{eq lgf def} where $\{h_n\}_{n\geq 1}$ is chosen to be a sequence of functions on $\m D$. The pairing of GFF with signed Borel measures $\rho$ on $\m D$ is defined analogously as in \eqref{eq lgf process def}.
\begin{itemize}
    \item The \textit{zero boundary (Dirichlet) GFF} on $\m D$ is given by choosing $\{h_n\}_{n\geq 1}$ to be any orthonormal basis of $H_0^1(\m D)$. The corresponding covariance kernel is given by the Dirichlet Green's function 
    \begin{equation} \label{eq Dirichlet Green function}
        G^{\mathrm{zero}}(z,w) = -\log|(z-w)/(1-z\bar w)|.
    \end{equation}
    We use $\G$ to denote a zero boundary GFF.
    
    \item We obtain the \textit{harmonic Gaussian field} on $\m D$ by choosing $\{h_n\}_{n\geq 1}$ to be any orthonormal basis of $\mc D_0$. The covariance kernel is given by 
    \begin{equation} \label{eq harmonic Green function}
        G(z,w) = -2\log |1-z\bar w|,
    \end{equation}
    which agrees with \eqref{eq lgf cov kernel} for $z,w\in \m S^1$. By the isomorphism between $\mc D_0$ and $\mc H_0$ given in Section~\ref{section:QS and WP}, we can identify it as the ``harmonic extension'' of an LGF to $\m D$. We will use $h$ to denote both an LGF on $\m S^1$ and its harmonic extension on $\m D$.

    \item The \textit{free boundary (Neumann) GFF on} $\m D$ (with mean zero on $\m S^1$) is obtained by choosing $\{h_n\}_{n\geq 1}$ to be any orthonormal basis of $H^1(\m D)/\m R$. The corresponding covariance kernel is given by the Neumann Green's function 
    \begin{equation}\label{eq Neumann Green function}
        G^{\mathrm{free}}(z,w) = -\log|(z-w)(1-z\bar w)|.
    \end{equation}
    We use $\overline \G$ to denote a free boundary GFF. 
\end{itemize}

\begin{remark}\label{rem:GFF decomposition}
    By the decomposition \eqref{eq Hilbert decomposition}, if $\G$ and $h$ are independent zero boundary GFF and harmonic Gaussian field on $\m D$, respectively, then $\G + h$ has the law of a free boundary GFF on $\m D$. This can also be seen from the identify $G^{\mathrm{zero}} + G = G^{\mathrm{free}}$ satisfied by the covariance kernels \eqref{eq Dirichlet Green function}, \eqref{eq harmonic Green function}, and \eqref{eq Neumann Green function}. We emphasize that the multiplicative factor of 2 in \eqref{eq harmonic Green function} (hence also in \eqref{eq lgf cov kernel}) is necessary for this relationship to hold.
\end{remark}
\subsection{Quasi-invariance of Log-correlated Gaussian field}
In this subsection, our goal is to culminate in the quasi-invariance result for pullbacks with respect to Weil-Petersson homeomorphisms (Theorem \ref{thm:main_1}).

Given $\varphi \in \Homeo_+(\m S^1)$, we define the pullback of an LGF $h = \sum_{n\geq 1} \xi_n h_n$ formally as 
\begin{equation}\label{eq def pullback lgf sum}
    h \circ \varphi: = \sum_{n\geq 1} \xi_n (h_n \circ \varphi).
\end{equation}
This can be rigorously considered as the centered Gaussian process 
\begin{equation}\label{eq def pullback lgf process}
    (h\circ \varphi, \rho) := (h, \varphi_*\rho)
\end{equation}
indexed by signed Borel measures $\rho$ on $\m S^1$ for which the right-hand side of \eqref{eq def pullback lgf process} is well-defined. Here, $\varphi_* \rho$ is the pushforward of $\rho$ under the homeomorphism $\varphi$. Note that the covariance kernel of $h\circ \varphi$ is 
\begin{equation}
    G_\varphi(e^{\ii \theta},e^{\ii \tilde \theta} ) := G(e^{\ii \theta}, e^{\ii \tilde \theta}) = -2\log\big| \varphi(e^{\ii \theta}) - \varphi(e^{\ii \tilde \theta}) \big|
\end{equation}
since
\begin{equation*}\begin{split}
    \COV\big((h\circ\varphi,\rho),(h\circ \varphi, \tilde \rho)\big) & =\COV\big((h,\varphi_* \rho),(h,\varphi_* \tilde \rho)\big)\\
    &= \iint_{\m S^1 \times \m S^1} G(e^{\ii \theta},e^{\ii \tilde \theta} )\, d(\varphi_* \rho)(e^{\ii \theta})\, \dd (\varphi_* \tilde \rho)(e^{\ii \tilde \theta})\\
    &= \iint_{\m S^1 \times \m S^1} G(e^{\ii \theta}, e^{\ii \tilde \theta})\, d\rho(e^{\ii \theta})\, \dd  \tilde \rho(e^{\ii \tilde \theta}).
    \end{split}
\end{equation*}

If $\varphi \in \QS(\m S^1)$, then we define the mean-zero part of the pullback $h\circ \varphi$ of the LGF $h = \sum_{n\geq 1} \xi_n h_n$ formally as
\begin{equation}\label{eq pullback field}
    \Pi(\varphi)(h) = \pi_0(h \circ \varphi) := \sum_{n\geq 1} \xi_n \Pi(\varphi)(h_n).
\end{equation}
For quasisymmetric $\varphi$, since $\Pi(\varphi)$ is a bounded linear operator on $\mc H_0$, the \eqref{eq pullback field} defines $\Pi(\varphi)(h)$ as a centered Gaussian field on $\m S^1$ with the Cameron--Martin space $\mc H_0$ and the \textit{covariance operator} $\Pi(\varphi)\Pi(\varphi)^*$. That is, 
$\{\langle \Pi(\varphi)(h),f\rangle = \sum_{n\geq 1} \xi_n \langle \Pi(\varphi)(h_n),f\rangle\}_{f\in \mc H_0}$ is a centered Gaussian process with 
\begin{equation}
    \COV(\langle \Pi(\varphi)(h),f\rangle,\langle \Pi(\varphi)(h),g\rangle) = \langle \Pi(\varphi)^*f,\Pi(\varphi)^*g\rangle = \langle f , \Pi(\varphi)\Pi(\varphi)^* g\rangle.
\end{equation}

\begin{remark}\label{remark mean-zero part}
    The relationship between $h\circ \varphi$ and $\Pi(\varphi)(h)$ is the following: if 
    \begin{equation}\label{eq pullback existence}
        \iint_{\m S^1 \times \m S^1} G_\varphi(e^{\ii \theta},e^{\ii \tilde \theta} ) \, \dd \t \, \dd \tilde\t < \infty
    \end{equation}
    so that $\fint_{\m S^1} h \circ \varphi := (h \circ \varphi, (2\pi)^{-1} \, \dd \theta)$ is an a.s.\ finite random variable, then 
    \begin{equation}
        (\Pi(\varphi)(h),\rho)  = (h\circ \varphi, \rho) - \rho(\m S^1) \fint_{\m S^1} h\circ \varphi
    \end{equation}
    a.s.\ for signed Borel measures $\rho$ on $\m S^1$ for which $\iint_{\m S^1\times \m S^1} |G_\varphi(e^{\ii \theta},e^{\ii \tilde \theta} ) |\, \dd\rho(e^{\ii \t})\,\dd\rho(e^{\ii \tilde \t})<\infty$. This can be seen directly from the decompositions \eqref{eq def pullback lgf sum} and \eqref{eq pullback field} for $h\circ \varphi$ and $\Pi(\varphi)(h)$, respectively. 
    
    One sufficient condition for \eqref{eq pullback existence} is for $\varphi$ to be a diffeomorphism, since then 
    \begin{equation}\label{eq log difference}
        u_\varphi(e^{\ii\t},e^{\ii\tilde\t}) := \log \abs{ \frac{\varphi(e^{\ii \t})-\varphi(e^{\ii\tilde\t})}{e^{\ii\t}-e^{\ii\tilde\t}} }
    \end{equation}
    is a continuous function on $\m S^1 \times \m S^1$ with $u_\varphi(e^{\ii \t},e^{\ii\t}) = \log|\varphi'(e^{\ii \t})|$. Hence, $G_\varphi = G - 2u_\varphi$ is integrable with respect to the arc-length measure on $\m S^1 \times \m S^1$. 
    
    More generally, it suffices for $\varphi^{-1}$ to be absolutely continuous with respect to the arc-length measure and satisfy $|(\varphi^{-1})'| \in H^{-1/2}(\m S^1)$. Then,
    \begin{equation}\label{eq constant term}
    \begin{split}
            \iint_{\m S^1 \times \m S^1} G_\varphi(e^{\ii\t},e^{\ii\tilde\t})\,\dd\t\,\dd\tilde\t &= \iint_{\m S^1 \times \m S^1} \left(\sum_{n\geq 1}(h_n\circ\varphi)(e^{\ii\t})\,(h_n\circ\varphi)(e^{\ii\tilde\t})\right) \dd\t \, \dd \tilde \t \\
            &= \sum_{n\geq 1} \left| \int_{\m S^1} h_n(\varphi(e^{\ii \t}))\,\dd\t \right|^2 = \sum_{n\geq 1} \left| \int_{\m S^1} h_n(e^{\ii \t})\,|(\varphi^{-1})'|(e^{\ii\t})\,\dd\t \right|^2 \\
            &= \||(\varphi^{-1})'|\|_{H^{-1/2}(\m S^1)/\m R}^2 < \infty.
            \end{split}
    \end{equation}
    In particular, if $\varphi \in \WP(\m S^1)$, then $|(\varphi^{-1})'| \in L^2(\m S^1)$; this follows from the standard argument using the VMO space that $\log|\varphi'| \in H^{1/2}(\m S^1) $ implies $\abs{\varphi'}^p = \exp(p\log\abs{\varphi'}) \in L^1(\m S^1)$ for any $p\geq 1$. Hence, we have \eqref{eq pullback existence} in this case.
\end{remark}

We now give our first main result, which identifies $\WP(\m S^1)$ as the class of quasisymmetric circle homeomorphisms $\varphi$ for which $\Pi(\varphi)(h)$ and $h$ have equivalent laws.

\begin{thm}\label{thm:main_1}
    Suppose $\varphi \in \QS(\m S^1)$ and $f$ is a fixed real-valued function on $\m S^1$. Let $\Pi(\varphi)(h)$ be the Gaussian field given in \eqref{eq pullback field}, where $h$ is an LGF on $\m S^1$. Then, the law of the random field $\Pi(\varphi)(h) + f$ is equivalent to that of an LGF on $\m S^1$ if and only if $\varphi \in \WP(\m S^1)$ and $f\in \mc H_0$. Otherwise, the two laws are mutually singular.
\end{thm}
\begin{proof}
    The Feldman--H\'ajek theorem states that the laws of $h$ and $\Pi(\varphi)(h) + f$ are either equivalent or mutually singular, and the former holds if and only if 
    \begin{itemize}
        \item The Cameron--Martin space for the law of $\Pi(\varphi)(h) + f$ is $\mc H_0$;
        \item The difference in the means, $f$, lies in the common Cameron--Martin space $\mc H_0$;
        \item The difference in the covariance operators, $\Pi(\varphi)\Pi(\varphi)^* - I$, is a Hilbert--Schmidt operator on the common Cameron--Martin space $\mc H_0$.
    \end{itemize}
    See, e.g., \cite[Thm.\ 2.25]{DPZ-sde-infinite-dim}. By Lemma \ref{lem properties of Pi_var}, these conditions are satisfied if and only if $\varphi \in \WP(\m S^1)$ and $f\in \mc H_0$.
    
    Recalling from Definition \ref{df:WP} that $\varphi \in \Homeo_+(\m S^1)$ is in the Weil--Petersson class if and only if $\log |\varphi'| \in H^{1/2}(\m S^1)$, we immediately obtain the following corollary from Theorem~\ref{thm:main_1}.
\end{proof}
\begin{cor}
    Let $\varphi \in \QS(\m S^1)$ and $Q$ be a real constant. Furthermore, let $u_\varphi$ be the function on $\m S^1 \times \m S^1$ given in \eqref{eq log difference} and let $z_0 \in \m S^1$, $\alpha \in \m R$ be fixed. If $h$ is an LGF on $\m S^1$, then the law of the random field 
\begin{equation*}
    \Pi(\varphi)(h)+\pi_0\left(Q\log\abs{\varphi'}+\alpha u_\varphi(\cdot,z_0)\right)
\end{equation*} is equivalent to that of $h$ if and only if $\varphi \in \WP(\m S^1)$ and either of the following holds: $\alpha=0$ or $u_\varphi(\cdot,z_0)\in H^{1/2}(\m S^1)$. Otherwise, the two laws are mutually singular.
\end{cor}

\subsection{Preliminaries of Gaussian multiplicative chaos}
Now we shift our attention to the Gaussian multiplicative chaos (GMC) measure on the unit circle. Given $\gamma \in (0,2)$, the $\gamma$-GMC measure with respect to an LGF $h$ on $\m S^1$ is defined formally as 
\begin{equation}\label{eq gmc formal}
    \mc M_h^\gamma(\dd\theta) = e^{\frac{\gamma}{2}h(e^{\ii \theta}) - \frac{\gamma^2}{8}\m E[h(e^{\ii \theta})^2]}\,\dd\theta.
\end{equation}
We omit the superscript $\gamma$ and write $\mc M_h = \mc M_h^\gamma$ if there is no confusion.
A rigorous study of GMC was initiated by Kahane \cite{Kahane} with seminal contributions by Robert and Vargas \cite{rv-gmc}, Duplantier and Sheffield \cite{DS11_LQG_KPZ}, and Shamov \cite{Shamov-GMC}. 
The GMC measure $\mc M_h$ is defined rigorously via renormalization. The following two equivalent approaches almost surely give the same measure as shown in \cite{berestycki}.
\begin{itemize}
    \item Let $\{h_k\}_{k\geq 1}$ be a sequence of continuous functions on $\m S^1$ forming an orthonormal basis for $\mc H_0$. Given an LGF $h = \sum_{k\geq 1} \xi_k h_k$, we define $\mc M_h^\gamma$ as the almost sure weak limit of the measures
    \begin{equation}\label{eq gmc def basis}
        \exp\left( \frac{\gamma}{2} \sum_{k=1}^n \xi_k h_k(e^{\ii \theta}) - \frac{\gamma^2}{8} \sum_{k=1}^n \big(h_k(e^{\ii \theta})\big)^2 \right) \dd \theta
    \end{equation}
    as $n\to \infty$.

    \item Let $\sigma$ be a fixed nonnegative Radon measure on the interval $(-\pi,\pi)$ with unit mass and $\sup_{x\in (-\pi,\pi)} \int_{-\pi}^\pi \log_+(1/ |x-y|) \,\sigma(\dd y) < \infty$. For $e^{\ii \theta} \in \m S^1$ and $\vare \in (0,1)$, define $\sigma_{e^{\ii \theta},\vare} $ to be the pushforward of $\sigma$ under the map $x \mapsto e^{\ii (\theta + \vare x)}$, and denote $h_\vare(e^{\ii \theta}) := (h, \sigma_{e^{\ii \theta},\vare})$. Then, we define $\mc M_h^\gamma$ as the weak limit as $\vare \to 0$ of 
    \begin{equation} \label{eq gmc def mollify}
        \exp \left( \frac{\gamma}{2} h_\vare(e^{\ii \theta}) + \frac{\gamma^2}{4} \iint_{\m S^1 \times \m S^1} \log |x-y| \,  \sigma_{e^{\ii \theta},\vare}(\dd x)\, \sigma_{e^{\ii \theta},\vare}(\dd y) \right)\dd \theta
    \end{equation}
    in probability.

    For instance, we can choose $\sigma$ to be the uniform probability measure on $[-1,1]$, in which case $h_\vare(e^{\ii \theta})$ will be average value of the LGF $h$ on the closed interval from $e^{\ii (\theta - \vare)}$ to $e^{\ii (\theta + \vare)}$ on $\m S^1$. However, the limit \eqref{eq gmc def mollify} does not depend on the choice of the measure $\sigma$.
\end{itemize}

Since the GMC measure $\mc M_h$ is almost surely determined by the LGF $h$, we can define the measure $\mc M_h$ using the limits \eqref{eq gmc def basis} or \eqref{eq gmc def mollify} when $h$ is a random field on $\m S^1$ whose law is absolutely continuous with respect to LGF. When $h$ is a random field on $\m S^1$ with a decomposition $h = c + \tilde h$ where $c$ is a real-valued random variable and the law of $\tilde h$ is absolutely continuous with respect to LGF, then we define 
\begin{equation}
    \mc M_h^\gamma := e^{(\gamma/2)c} \mc M_{\tilde h}^\gamma.
\end{equation}

Conversely, the GMC measure $\mc M_h^\gamma$ almost surely determines the LGF $h$ \cite[Thm.~1.4]{bss-equiv-gmc}.

\begin{remark}\label{remark critical gmc limit}
When $\gamma=2$, the limits \eqref{eq gmc def basis} and \eqref{eq gmc def mollify} give the zero measure almost surely. However, a slightly modified limit gives a nontrivial measure $\mc M_h^{\mathrm{crit}}$ called the  \textit{critical Gaussian multiplicative chaos} \cite{DRSV-critical1,DRSV-critical2}. 
As described in \cite[Section~4.1.2]{APS-crit-gmc} (also see \cite{Powell-crit-review} and \cite{PS-quantum-length}), we can also obtain the critical GMC measure through the weak limit $\frac{1}{2(2-\gamma)}\mc M_h^{\gamma} \to \mc M_h^{\mathrm{crit}}$ in probability as $\gamma\to 2^{-}$. We also have that the critical GMC measure $\mc M_h^{\mathrm{crit}}$ almost surely determines the LGF $h$ \cite{vihko-equiv-gmc}. 
\end{remark}

\begin{remark}
Our multiplicative factor of $\gamma/2$ in \eqref{eq gmc def basis}--\eqref{eq gmc def mollify} as well as the factor of 2 in \eqref{eq lgf cov kernel} differs from some of the recent works on the GMC on the circle, such as \cite{CN-gmc-cbe,GV-gmc-fourier}.
In particular, they are chosen to agree with the literature on Liouville quantum gravity \cite{DS11_LQG_KPZ}. In the LQG theory, the quantum boundary length of a two-dimensional domain $D$ is defined as the GMC measure on $\partial D$ with respect to the free boundary GFF on $D$. More precisely, assume $D \subset \m H$ and $\partial D \cap \m R$ is a nonempty interval. Given free boundary GFF $\overline \Gamma = \Gamma + h$ where $\Gamma$ is a zero boundary GFF and $h$ is an independent harmonic Gaussian field $h$ on $D$ (recall Remark \ref{rem:GFF decomposition}), we define the $\gamma$-LQG boundary length $\nu_{\overline \G}$ on $\partial D \cap \m R$ as the almost sure weak limit 
\begin{equation}
    \lim_{\vare \to 0} \vare^{\gamma^2/4} e^{\frac{\gamma}{2} \overline \G_\vare(x)}\, \dd x,
\end{equation}
where $\overline \Gamma_\vare(x)$ is the average value of $\overline \G$ on the semicircle $\partial B_\vare(x) \cap \m H$. By the equivalence between the GMC defined using (semi)circle averages and the Karhunen--Lo\`eve expansion of GFF as explained in \cite{berestycki}, we have that 
\begin{equation}
    \lim_{\vare \to 0} e^{\frac{\gamma}{2} \overline \G_\vare(x) - \frac{\gamma^2}{8}\m E[\overline \G_\vare(x)^2]}\, \dd x = \lim_{\vare \to 0} e^{\frac{\gamma}{2} h_\vare(x) - \frac{\gamma^2}{8}\m E[h_\vare(x)^2]}\, \dd x
\end{equation}
and thus
\begin{equation}\label{eq gmc 1d 2d difference}
    \nu_{\overline \Gamma} = \lim_{\vare \to 0} \vare^{\gamma^2/4} e^{\frac{\gamma}{2} h_\vare(x) - \frac{\gamma^2}{8}\m E[\G_\vare(x)^2]}\, \dd x = e^{-(\gamma^2/8)C} \nu_h.
\end{equation}
Here, 
\begin{equation}\begin{split}
    C = \lim_{\vare \to 0} \m E[\G_\vare(x)^2] &= \lim_{\vare \to 0} \frac{1}{\pi^2}\iint_{[0,\pi]^2} \big( -\log |\vare e^{\ii \theta} - \vare e^{\ii \tilde \theta}| + \log |\vare e^{\ii \t} - \vare e^{-\ii \tilde \t}| \big) \dd \t \, \dd \tilde \t  \\ 
    &= \frac{2}{\pi^2}\iint_{[0,\pi]^2}\log | e^{\ii \t} -  e^{-\ii \tilde \t}| \, \dd \t \, \dd \tilde \t = -\frac{7}{\pi^2}\zeta(3)  \approx -0.85,
    \end{split}
\end{equation}
where $\zeta(3)$ is Ap\'ery's constant.
By the isomorphism between $\mc H_0$ and the harmonic extension $\mc D_0$, the relation \eqref{eq gmc 1d 2d difference} holds with $\nu_h$, the boundary LQG measure on $\partial \m D$ defined using the harmonic field $h$ in the unit disk, replaced by $\mc M_h$, the GMC measure on $\m S^1$ defined using the LGF $h$ on the unit circle. In particular, if $\overline \G$ is a free boundary GFF on $\m D$ and $h$ is its trace on $\m S^1$ (equivalently, the harmonic part of $\overline \G$ on $\m D$), then the normalized measures $(\nu_{\overline \G}(\m S^1))^{-1} \nu_{\overline \G}$ and $(\mc M_h(\m S^1))^{-1} \mc M_h$ agree almost surely.
\end{remark}
\subsection{Quasi-invariance of Gaussian multiplicative chaos}
Given an LGF $h$ on $\m S^1$, define the corresponding \textit{normalized GMC measure} as 
\begin{equation}\label{eq normalized gmc}
    \widehat{\mc M}_h^\g := \frac{1}{\mc M_h(\m S^1)} \mc M_h^\g.
\end{equation}
Our goal for this subsection is to show the following quasi-invariance result for normalized GMC measures.
\begin{prop}\label{prop:quasi-invariance of weighted Liouville measure}
Suppose $\varphi \in \Homeo_+(\m S^1)$.
    Let $h$ be the LGF on $\m S^1$ and $\widehat{\mc M}_h=\widehat{\mc M}_h^\g$ be the corresponding normalized GMC measure \eqref{eq normalized gmc} for $\gamma \in (0,2]$.  Then, the law of the pullback $\varphi^* \widehat{\mc M}_h$ is absolutely continuous with respect to the law of $\widehat{\mc M}_h$ if and only if $\varphi \in \WP(\m S^1)$. Moreover, if $\varphi \in \WP(\m S^1)$, then we almost surely have
    \begin{equation}\label{eq:coordinate-change}
        \varphi^* \mc M_h ^{\g} = \mc M_{h\circ \varphi+Q\log\abs{\varphi'}} ^{\g}
    \end{equation}
    where $Q = \frac{2}{\g} + \frac{\gamma}{2}$.
\end{prop}

It is well-known that given a 2D log-correlated Gaussian field $\Gamma$, the corresponding LQG area measure $\mu_\Gamma$ almost surely satisfies the conformal coordinate change rule
\begin{equation}
    \psi^* \mu_\Gamma = \mu_{\Gamma\circ \psi + Q\log |\psi'|}
\end{equation}
if $\psi$ is a conformal map and $Q = 2/\gamma + \gamma/2$ \cite[Prop.~2.1]{DS11_LQG_KPZ}. In showing \eqref{eq:coordinate-change}, we extend this result to non-smooth and even non-$C^1$ homeomorphisms $\varphi$ on $\m S^1$.
To begin our proof of Proposition \ref{prop:quasi-invariance of weighted Liouville measure}, we verify that the coordinate change formula holds when $\varphi$ is a diffeomorphism. 
\begin{lem}\label{lem:diff-coord-change}
    Let $h$ be an LGF on the unit circle $\m S^1$ and suppose $\varphi$ is an orientation-preserving $C^1$ diffeomorphism of $\m S^1$. Then, for every $\g \in (0,2]$ with $Q = \frac{2}{\g} + \frac{\gamma}{2}$, we almost surely have the coordinate change rule \eqref{eq:coordinate-change}.
\end{lem}

We note that there exist $C^1$ diffeomorphisms $\varphi$ of $\m S^1$ that does not belong to the Weil--Petersson class, in which case the law of the mean-zero part of $h\circ \varphi + Q\log |\varphi'|$ is singular to that of LGF on $\m S^1$. Nevertheless, this is a Gaussian field whose covariance kernel $G_\varphi(x,y)$ is of the form $-\log |x-y| +  u_\varphi(x,y) $ where $u_\varphi$ is a continuous function on $\m S^1 \times \m S^1$ (recall Remark \ref{remark mean-zero part}). Hence, as proved in \cite{berestycki}, the GMC measure $\mc M_{h\circ\varphi + Q\log|\varphi'|}^\gamma$ is well-defined via the limit \eqref{eq gmc def mollify} where the mollified field $h_\vare(e^{\ii \theta}) = (h, \sigma_{e^{\ii \theta},\vare})$ is replaced with $ (h \circ \varphi, \sigma_{e^{\ii \theta},\vare}) + \int_{\m S^1} Q\log|\varphi'|\, \dd\sigma_{e^{\ii \theta},\vare}$.

\begin{proof}[Proof of Lemma \ref{lem:diff-coord-change}]
    Let us first assume that $\gamma \in (0,2)$. Consider the GMC measure given by the following weak limit in probability:
    \begin{equation*}
        \widetilde{\mc M}_{h\circ \varphi}^\gamma := \lim_{\vare\to 0} \exp \left( \frac{\gamma}{2} (h\circ \varphi)_\vare(e^{\ii \theta}) - \frac{\gamma^2}{8} \var\big((h\circ \varphi)_\vare(e^{\ii \theta})\big) \right)\, \dd\theta,
    \end{equation*}
    where 
    \begin{equation*}
        (h\circ \varphi)_\vare(e^{\ii \theta}):= \frac{1}{2\vare} \int_{-\vare}^{\vare} (h \circ \varphi)(e^{\ii (\theta + t)})\,\dd t.
    \end{equation*}
    Let us define $h_\vare(e^{\ii \t})$ similarly, so that $\var\big(h_\vare(e^{\ii \t})\big) = (2\vare)^{-2} \int_{-\vare}^\vare \int_{-\vare}^\vare -2\log|e^{\ii \t} - e^{\ii \tilde \t}|\, \dd \t\, \dd \tilde\t$. Then,
    \begin{equation*}
        \var\big((h\circ \varphi)_\vare(e^{\ii \theta})\big) - \var\big(h_\vare(e^{\ii \t})\big) = \frac{1}{4\vare^2} \int_{-\vare}^\vare \int_{-\vare}^\vare -2 \log  \abs{\frac{\varphi(e^{\ii (\theta+t)}) - \varphi(e^{\ii (\theta+s)})}{e^{\ii (\theta+t)}-e^{\ii (\theta+s)}}}\, \dd s \, \dd t 
    \end{equation*}
    converges uniformly to $-2\log|\varphi'(e^{\ii \t})|$ as $\vare\to 0$. Hence, we have
    \begin{equation}
        \frac{\dd \widetilde{\mc M}_{h \circ\varphi}^\gamma}{\dd \mc M_{h \circ \varphi}^\gamma} (e^{\ii \theta}) = \big|{\varphi'(e^{\ii \theta})}\big|^{\gamma^2/4} .
    \end{equation}
    
    Let us choose $\{h_n\}_{n\geq 1}$ to be the orthonormal basis \eqref{eq trigonometric basis} for $\mc H_0$. Suppose we have the decomposition $h = \sum_{n\geq 1} \xi_n h_n$, where $\{\xi_n\}_{n\geq 1}$ are i.i.d.\ standard normal random variables. Then, since $h \circ\varphi = \sum_{n\geq 1} \xi_n (h_n \circ \varphi)$ is a Karhunen--Lo\`eve expansion where each $h_n \circ \varphi$ are continuous on $\m S^1$, we have from \cite[Sec.~5]{berestycki} that for any interval $A \subset \m S^1$, 
    \begin{equation}\begin{split}
        \mc M_h^\gamma(\varphi(A)) &= \lim_{N\to \infty} \int_{\varphi(A)} e^{\sum_{n=1}^N \big( \frac{\gamma}{2}\xi_n h_n(e^{\ii \theta}) - \frac{\gamma^2}{8} [h_n(e^{\ii \theta})]^2 \big) }\,\dd\theta\\
        &= \lim_{N\to \infty} \int_{A} e^{\sum_{n=1}^N \big( \frac{\gamma}{2}\xi_n (h_n\circ\varphi)(e^{\ii \theta}) - \frac{\gamma^2}{8} [(h_n\circ\varphi)(e^{\ii \theta})]^2 \big) }|\varphi'(e^{\ii \t})|\,\dd\theta\\
        &=\int_A |\varphi'(e^{\ii\theta})|\, \dd \widetilde{\mc M}_{h\circ\varphi}^\gamma (e^{\ii \theta}) = \int_A |\varphi'(e^{\ii\theta})|^{1+\gamma^2/4}\, \dd {\mc M}_{h\circ\varphi}^\gamma (e^{\ii \theta}) \\
        &= \mc M_{h \circ \varphi + Q\log|\varphi'|}^\gamma(A)
    \end{split}\end{equation}
    almost surely. For the last equality, we used the identity $\dd\mc M_{h + f}^\gamma/\dd \mc M_h^\gamma = e^{(\gamma/2)f}$ for any continuous $f$ on $\m S^1$, which can be checked from the definition \eqref{eq gmc def mollify} of the GMC measure. The equality for critical GMC holds by taking the limit as $\gamma \to 2^-$ as in Remark \ref{remark critical gmc limit}.
\end{proof}
 To extend Lemma \ref{lem:diff-coord-change} to the Weil--Petersson class, we approximate $\varphi$ by a sequence of diffeomorphisms which are continuous in the topology induced by the Weil--Petersson metric as in Lemma \ref{lem:wp-approx}.
   We now show that the GMC measure on $\m S^1$ is almost surely continuous under these approximations.  To our knowledge, this is the first time that the GMC measure is shown to be continuous with respect to the random field. The continuity with respect to the base measure and the parameter $\gamma$ has been shown in 2D in \cite[Prop~4.1]{PS-quantum-length}.
   
\begin{lem}\label{lem:gmc-convergence}
    Let $A \in \ms B(\mc H_0)$ be invertible in $\ms B(\mc H_0)$ and suppose $AA^*-I$ is Hilbert--Schmidt. Suppose $\{A_n\}_{n\geq 1}$ is a sequence in $\ms B(\mc H_0)$ such that $A_n$ converges to $A$ in the weak operator topology and $\|(A_nA^{-1})(A_nA^{-1})^*-I\|_{\HS} \to 0$. Let $\{\beta_n\}_{n\geq 1}$ be a sequence in $\mc H_0$ which converges to $\beta \in \mc H_0$. Let $h$ be an LGF on $\m S^1$. Then, for $\gamma \in (0,2)$, the Gaussian multiplicative chaos measures $\mc M_{A_n h + \beta_n}$ converges weakly to $\mc M_{Ah + \beta}$ in probability.
\end{lem}

\begin{proof}
    Assume first that $A = I$ and $\beta = 0$. In this case, as outlined in \cite[Sec.~6]{berestycki}, it suffices to show that $\mc M_{A_nh+\beta}(S) \to \mc M_h(S)$ in probability for each fixed interval $S \subseteq \m S^1$. The proof proceeds by considering the approximations \eqref{eq gmc def basis} of $X:=\mc M_h(S)$ using the orthonormal basis $\{h_k\}_{k\in \m N}$ for $\mc H_0$ given in \eqref{eq trigonometric basis}. That is, let 
    \begin{equation}
        X_m:= \int_S \exp\left( \frac{\gamma}{2}\sum_{k=1}^{m} \langle h, h_k\rangle h_k(e^{\ii\theta}) -\frac{\gamma^2}{8} \sum_{k=1}^m (h_k(e^{\ii\theta}))^2 \right)\dd\theta.
    \end{equation}
    We know from \cite[Eq.~(5.1)]{berestycki} that $X_m = \m E[X|\mc F_m]$ where $\mc F_m$ is the $\sigma$-algebra generated by $\{\langle h,h_k\rangle\}_{1\leq k\leq m}$, whence $X_m\to X$ in $L^1(\m P)$. We consider the analogous approximations 
    \begin{equation}
        X_m^{(n)}:= \int_S \exp\left( \frac{\gamma}{2}\sum_{k=1}^{m} \langle A_nh+\beta_n, h_k\rangle h_k(e^{\ii\theta}) -\frac{\gamma^2}{8} \sum_{k=1}^m (h_k(e^{\ii\theta}))^2 \right)\dd\theta,
    \end{equation}
    which converges in probability to $X^{(n)}:= \mc M_{A_nh+\beta_n}(S)$ as $n\to\infty$ since the law of $A_nh + \beta_n$ is absolutely continuous with respect to that of $h$ by the Feldman--H\'ajek theorem (cf.\ Theorem \ref{thm:main_1}). We now claim the following.
    \begin{itemize}
        \item For each $m\in \m N$, as $n\to\infty$,
        \begin{equation}\label{eq:approx1}
            X_m^{(n)} \to X_m \quad \text{in probability}.
        \end{equation}
        \item There exists a positive integer $n_0$ such that for any $\delta>0$, we have 
        \begin{equation}\label{eq:approx2}
            \lim_{m\to\infty} \sup_{n\geq n_0} \m P(|X_m^{(n)}-X^{(n)}|>\d) = 0.
        \end{equation}
    \end{itemize}
    Then, for any given $\d>0$, we can choose sufficiently large $m$ and then $n_0$ such that 
    \begin{equation*}
        \m P(|X^{(n)} - X_m^{(n)}| > \d) +    \m P(|X_m^{(n)} - X_m|>\d) + \m P(|X_m - X|>\d) < \d.
    \end{equation*}
    That is, $X^{(n)} \to X$ in probability as $n\to\infty$.
    
    Let us now show \eqref{eq:approx1}. Note from \eqref{eq trigonometric basis} that $|h_k(e^{\ii\theta})|\leq \sqrt 2$ for every $k$ and $\theta$. Hence, 
    \begin{equation*}
        \begin{split}
            \m E \left[ \sup_{e^{\ii\theta}\in \m S^1} \left|\sum_{k=1}^m\langle (A_n-I)h + \beta_n, h_k\rangle h_k (e^{\ii\theta})\right| \right] 
            &\leq \sqrt 2\sum_{k=1}^{m} \big(\m E\left[ |\langle (A_n-I)h,h_k\rangle| \right]+ |\langle\beta_n,h_k\rangle|\big)\\
            &\leq \sqrt 2\sum_{k=1}^{m} \norm{(A_n-I)^*h_k}_{\mc H_0} + 2\sqrt{m} \norm{\beta_n}_{\mc H_0}.
        \end{split}
    \end{equation*}
    Note that
    \begin{equation*}
        \begin{split}
            \norm{(A_n-I)^*h_k}_{\mc H_0}^2 &= |\langle h_k, (A_nA_n^*-I) h_k\rangle +\langle h_k, (2I-A_n-A_n^*)h_k\rangle| \\
            &\leq \norm{A_nA_n^*-I}_{\HS} + 2|\langle h_k, (A_n-I)h_k\rangle|,
        \end{split}
    \end{equation*}
    which tends to 0 as $n\to \infty$ by assumption. Combining the above two inequalities, we see that given any sequence of positive integers, we can find a further sequence $\{n_j\}_{j\geq 1}$ such that 
    $   \sup_{e^{\ii\theta}\in \m S^1} |\sum_{k=1}^m \langle (A_{n_j}-I)h + \beta_{n_j}, h_k\rangle - \langle h, h_k\rangle h_k (e^{\ii\theta})| \to 0 $
    almost surely as $j\to\infty$. As we have 
    \begin{equation*}
      \inf_{e^{\ii\theta}\in S} e^{\frac{\g}{2} \sum_{k=1}^m \langle (A_{n_j}-I) h + \beta_{n_j}, h_k\rangle h_k(e^{\ii\theta})} \leq  \frac{X_m^{(n_j)}}{X_m} \leq \sup_{e^{\ii\theta}\in S} e^{\frac{\g}{2} \sum_{k=1}^m \langle (A_{n_j}-I) h + \beta_{n_j}, h_k\rangle h_k(e^{\ii\theta})} 
    \end{equation*}
    from the definitions of $X_m^{(n)}$ and $X_m$, we conclude that $X_m^{(n_j)} \to X_m$ almost surely as $j\to \infty$ and thus $X_m^{(n)} \to X_m$ in probability as $n\to\infty$. 

    Let us now show \eqref{eq:approx2}. For this, we will first check that if $\rho_n$ is the Radon--Nikodym derivative of the law of $A_n h + \beta_n$ with respect to that of $h$, then $\sup_{n\geq n_0} \m E[(\rho_n)^2]<\infty$ for sufficiently large $n_0$.\footnote{The same calculations can be used to show that $\sup_{n\geq 1} \m E[(\rho_n)^p]<\infty$ for sufficiently small $p>1$, or, for any given $p>1$, we have $\sup_{n\geq n_0} \m E[(\rho_n)^p]<\infty$ for sufficiently large $n_0$.} Let us fix $A \in \ms B(\mc H_0)$ with $\|AA^*-I\|_{\HS}^2<1/2$ and $\beta \in \mc H_0$. Choose $\{\tilde h_n\}_{n\geq 1}$ be an orthonormal eigenbasis for $\mc H_0$ with respect to $AA^*-I$, with $(AA^*-I)\tilde h_k = a_k \tilde h_k$ for each $k$. Observe that $\{\langle h,\tilde h_k\rangle\}_{k\geq 1}$ are i.i.d.\ standard normal random variables, whereas $\{\langle Ah + \beta, \tilde h_k\rangle\}_{k\geq 1}$ are independent Gaussian with mean $\langle \beta,\tilde h_k\rangle=:b_k$ and variance $\langle \tilde h_k, AA^*\tilde h_k\rangle = (1+a_k)$. Hence, the Radon--Nikodym derivative $\rho$ of the law of $Ah+\beta$ with respect to that of $h$ is given by 
    \begin{equation*}
        \rho\left( \sum_{k\geq 1} x_k \tilde h_k \right) = \prod_{k\geq 1} \frac{(2\pi(1+a_k))^{-1/2} \exp(-(x_k-b_k)^2/2(1+a_k))}{(2\pi)^{-1/2}\exp(-x_k^2/2)} =:\prod_{k\geq 1} F_k(x_k).
    \end{equation*}
    A straightforward calculation gives us that if $Z$ is a standard normal random variable, then 
    \begin{equation*}
        \m E[\rho^2] = \prod_{k\geq 1} \m E[(F_k(Z))^2] = \prod_{k\geq 1} \frac{1}{(1-a_k^2)^{1/2}} \exp\left(\frac{b_k^2}{1-a_k^2}\right).
    \end{equation*}
    Now, since $\sum_{k\geq 1} a_k^2 = \|AA^*-I\|_{\HS}^2 < 1/2$, we have $a_k^2 < 1/2$ for all $k$. Hence,
    \begin{equation*}
        \log \m E[\rho^2] \leq \sum_{k\geq 1} a_k^2 + 2\sum_{k\geq 1}b_k^2 = \|AA^*-I\|_{\HS}^2 + 2\|\beta\|_{\mc H_0}^2.
    \end{equation*}
    Since $\|A_nA_n^*-I\|_{\HS}\to 0$ and $\|\beta_n\|_{\mc H_0}\to 0$ as $n\to\infty$, we conclude that $\m E[\rho_n^2] \to 1$ as $n\to\infty$. Fix $n_0$ to be a positive integer such that $\sup_{n\geq n_0} \m E[(\rho_n)^2]<\infty$. Now, given any $\d>0$, we have
    \begin{equation*}
        \begin{split}
            \m P(|X_m^{(n)} - X^{(n)}| > \delta) = \m E[\rho_n \mathbf 1_{\{|X_m - X| > \delta\}}]
            &\leq K \m P(|X_m - X|>\delta) + \m E[\rho_n \mathbf 1_{\{\rho_n> K\}}] \\
            &\leq \frac{K}{\d} \m E|X_m-X| + \sqrt{\frac{\m E[(\rho_n)^2]}{K}}.
        \end{split}
    \end{equation*}
    Since $X_m \to X$ in $L^1(\m P)$, by first choosing sufficiently large $K$ and then choosing large $m$, we can make $\sup_{n\geq n_0} \m P(|X_m^{(n)} - X^{(n)}|>\delta)$ arbitrarily small.
    This completes the proof of \eqref{eq:approx2} and therefore the lemma in the case $A=I$ and $\beta=0$.

    For general $A\in \ms B(\mc H_0)$ such that $\|AA^*-I\|_{\HS}<\infty$ and $\beta\in \mc H_0$, recall that the law of $\tilde h:=Ah + \beta$ is absolutely continuous with respect to that of $h$ by the Feldman--H\'ajek theorem. Let  $\tilde A_n := A_nA^{-1} \in \ms B(\mc H_0)$ and $\tilde \beta_n:= \beta_n - A_nA^{-1} \beta \in \mc H_0$, so that $A_nh+\beta_n = \tilde A_n \tilde h + \tilde \beta_n$. Then, $\tilde A_n \to I$ in the weak operator topology, $\|\tilde A_n \tilde A_n^*-I\|_{\HS} \to 0$, and $\|\tilde \beta_n\|_{\mc H_0}\to 0$ by the assumptions of the lemma, so we conclude that $\mc M_{\tilde A_n\tilde h + \tilde \beta_n} = \mc M_{A_n h + \beta_n}$ converges weakly to $\mc M_{\tilde h} = \mc M_{Ah + \beta}$ in probability.
\end{proof}

We are now ready for a proof of Proposition \ref{prop:quasi-invariance of weighted Liouville measure}. We divide this into two parts, first showing that if $\varphi \in \WP(\m S^1)$, then the laws of $\widehat{\mc M}_h$ and $\widehat{\mc M}_{h\circ \varphi + Q\log|\varphi'|}$ are absolutely continuous by checking that the coordinate change formula \eqref{eq:coordinate-change} holds.

\begin{proof}[Proof of Sufficiency] 
Let us first consider $\gamma \in (0,2)$.
    Suppose $\varphi \in \WP(\m S^1)$ and let $\varphi_n \in \Diff_+(\m S^1)$ be a sequence of approximations to $\varphi$ given in Lemma~\ref{lem:wp-approx}. Let $A_n = \Pi(\varphi_n)$ and $\beta_n = \pi_0( Q\log |\varphi_n'|)$ for each $n$, and similarly let $A = \Pi(\varphi)$ and $\beta = \pi_0(Q\log|\varphi'|)$. Then, we have $A_n \to A$ in the strong operator topology and $\|(A_nA^{-1})(A_nA^{-1})^*-I\|_{\HS}\to 0$ by our choice of approximations. Also, since $\log|(\varphi_n \circ \varphi^{-1})'| = (\log |\varphi_n'| - \log|\varphi'|) \circ \varphi^{-1}$, we have $\|\beta_n - \beta\|_{\mc H_0} \leq Q\|\Pi(\varphi)\|_{\ms B(\mc H_0)} \|\log|(\varphi_n \circ \varphi^{-1})|'\|_{\mc H_0}$ tending to 0 as $n\to\infty$.

    We have from Lemma~\ref{lem:diff-coord-change} that the coordinate change formula
    \begin{equation*}
        \varphi_n^* \mc M_h = \mc M_{h \circ \varphi_n + Q\log |\varphi_n'|} 
        = e^{\frac{\gamma}{4\pi}\int_{\m S^1} (h \circ \varphi_n + Q\log |\varphi_n'|)} \mc M_{A_nh + \beta_n}
    \end{equation*}
    holds almost surely for every $n$. By Lemma~\ref{lem:gmc-convergence}, $\mc M_{A_nh + \beta_n}$ converges weakly in probability to $\mc M_{Ah + \beta}$. Since $\sum_k\abs{\int_{\m S^1}(h_k \circ\varphi_n - h_k \circ \varphi) }^2\leq C\norm{\varphi_n'-\varphi'}_{H^{-1/2}(\m S^1)}^2\to 0$ as $n\to\infty$, we have $\int_{\m S^1} h\circ \varphi_n \to \int_{\m S^1} h \circ\varphi$ almost surely. We also have $\int_{\m S^1} \log|\varphi_n'| \to \int_{\m S^1} \log|\varphi'|$, so 
    \begin{equation}
        \varphi_n^* \mc M_h \to e^{\frac{\gamma}{4\pi}\int_{\m S^1} (h \circ \varphi + Q\log |\varphi'|)} \mc M_{Ah + \beta} = \mc M_{h\circ\varphi + Q\log|\varphi'|}
    \end{equation}
    weakly in probability as $n\to\infty$.

    On the other hand, since $\|\varphi - \varphi_n\|_\infty \to 0$, for every $F\in C(\m S^1)$, we have 
    \begin{equation*} \begin{aligned}
        \lim_{n\to \infty} \int_{\m S^1} F\, \dd(\varphi_n^* \mc M_h) &= \lim_{n\to \infty} \int_{\m S^1} (F\circ \varphi_n)\, \dd\mc M_h = \int_{\m S^1} (F\circ \varphi)\, \dd\mc M_h = \int_{\m S^1} F\, \dd (\varphi^* \mc M_h)
    \end{aligned} \end{equation*}
    almost surely. That is, $\varphi_n^* \mc M_h$ converges almost surely in the weak topology to $\varphi^* \mc M_h$. We thus conclude that $\varphi^* \mc M_h = \mc M_{h\circ\varphi + Q\log|\varphi'|}$ almost surely. 
    As described in Remark \ref{remark critical gmc limit}, we have $\frac{1}{2(2-\gamma)}\mc M_h^{\gamma} \to \mc M_h^{\mathrm{crit}}$ weakly in probability as $\gamma\to 2^{-}$. Hence, the coordinate change rule \eqref{eq:coordinate-change} holds for the critical case $\gamma= 2$ as well.

    To conclude the proof, observe that  
    \begin{equation}
        \varphi^* \widehat{\mc M}_h = \frac{1}{\mc M_h(\m S^1)} \varphi^*\mc M_h = \frac{1}{\varphi^*\mc M_h(\m S^1)} \varphi^*\mc M_h = \frac{1}{\mc M_{Ah + \beta}(\m S^1)} \mc M_{Ah+\beta} = \widehat{\mc M}_{Ah+\b}
    \end{equation}
    almost surely and the law of $Ah+\b$ is equivalent to that of $h$ by Theorem \ref{thm:main_1}. Therefore, the law of $\varphi^* \widehat{\mc M}_h$ is equivalent to that of $\widehat{\mc M}_h$.
\end{proof}

\begin{proof}[Proof of necessity]
    Let $\varphi \in \Homeo_+(\m S^1)$ and assume that the law of the pull-back $\varphi^* \widehat{\mc M}_h$ is absolutely continuous with respect to that of $\widehat{\mc M}_h$. Then, since the GMC almost surely determines the LGF \cite{bss-equiv-gmc,vihko-equiv-gmc}, there exists a random mean-zero distribution $\tilde h$ on $\m S^1$ with a law absolutely continuous respect to LGF on $\m S^1$ such that $\widehat{\mc M}_{\tilde h} = \varphi^* \widehat{\mc M}_h$ almost surely.
    
    Choose a sequence of smooth $\varphi_n \in \Diff_+(\m S^1)$ which converges uniformly to $\varphi$ (e.g., by convolution). Since $\widehat{\mc M}_{h\circ \varphi_n + Q\log |\varphi_n'|}=\varphi_n^*\widehat{\mc M}_h \to \varphi^* \widehat{\mc M}_h$ almost surely, we have $\Pi(\varphi_n) h + \pi_0(Q\log|\varphi_n'|) \to \tilde h$ in law. In particular, $\tilde h$ is a Gaussian field whose law is absolutely continuous with respect to that of LGF on $\m S^1$. Then, by the Feldman--H\'ajek theorem, there exists some $Q \in \ms B(\mc H_0)$ with $Q-I$ Hilbert--Schmidt such that for any $f,g\in \mc H_0$,
    \begin{equation}\label{eq:nec-cov-conv}
    \begin{aligned}
        \lim_{n\to\infty} \langle\Pi(\varphi_n)^*f, \Pi(\varphi_n)^* g\rangle
        &= \lim_{n\to\infty} \mathrm{Cov}(\langle\Pi(\varphi_n) h, f\rangle, \langle\Pi(\varphi_n) h, g\rangle )\\
        &=\mathrm{Cov}(\langle\tilde h, f\rangle,\langle\tilde h, g\rangle) = \langle f, Qg\rangle
    \end{aligned} \end{equation}
    where $\langle\cdot,\cdot\rangle$ is the inner product on $\mc H_0$.
    That is, $\Pi(\varphi_n) \Pi(\varphi_n)^*$ converges in the weak operator topology to $Q$.
    In particular, for any $f\in \mc H_0$, we have 
    \begin{equation*}
        \lim_{n\to\infty} \|\Pi(\varphi_n)^* f\|_{\mc H_0}^2 = \langle f,Qf\rangle \leq \|Q\|_{\ms B(\mc H_0)} \|f\|_{\mc H_0}^2.
    \end{equation*}
    Then, by the uniform boundedness principle, 
    \begin{equation*}
        \sup_n \|\Pi(\varphi_n)\|_{\ms B(H)} = \sup_n \|\Pi(\varphi_n)^*\|_{\ms B(H)} < \infty.
    \end{equation*}
    For any $f\in \mc H_0$ with $\|f\|_{\mc H_0} \leq 1$, by Fatou's lemma,
    \begin{equation*}
        \begin{aligned}
            \|\Pi(\varphi) f\|_{\mc H_0}^2 &= \iint_{\m S^1 \times \m S^1} \frac{|f(\varphi(e^{\ii\theta})) - f(\varphi(e^{\ii \psi}))|^2}{|e^{\ii\theta} - e^{\ii \psi}|^2}\,\dd \t\,\dd \psi \\
            &\leq \liminf_{n\to\infty} \iint_{\m S^1 \times \m S^1} \frac{|f(\varphi_n(e^{\ii\theta})) - f(\varphi_n(e^{\ii \psi}))|^2}{|e^{\ii\theta} - e^{\ii \psi}|^2}\,\dd \t\,\dd \psi = \liminf_{n\to\infty} \|\Pi(\varphi_n)f\|_{\mc H_0}^2 \\
            &\leq \liminf_{n\to\infty} \|\Pi(\varphi_n)\|_{\ms B(\mc H_0)}^2 < \infty .
        \end{aligned}
    \end{equation*}
    Therefore, $\Pi(\varphi) \in \ms B(\mc H_0)$ and $\varphi \in \QS(\m S^1)$ follows by Lemma~\ref{lem properties of Pi_var}. 
    
    We see from the formulas \eqref{eq:op-entry-m} and \eqref{eq:op-entry-n} that $\Pi(\varphi_n) \to \Pi(\varphi)$ entrywise with respect to the basis \eqref{eq:basis} for $\mc H_0^{\m C}$. Hence, $Q = \Pi(\varphi) \Pi(\varphi)^*$ and $\Pi(\varphi) \Pi(\varphi)^*-I$ is Hilbert--Schmidt. We conclude $\varphi \in \WP(\m S^1)$ using Lemma~\ref{lem properties of Pi_var}.
\end{proof}

In the Liouville theory, we often consider Gaussian multiplicative chaos measures with respect to a log-correlated Gaussian field plus a logarithmic singularity (cf.\ Lemma~\ref{Int lem:welding-homeo}). Below, we give an analog of Propositon \ref{prop:quasi-invariance of weighted Liouville measure} for such fields. Recall from \eqref{eq:diff ratio} the log-ratio 
\begin{equation*}
    u_\varphi(z,w) = \log \frac{\varphi(z)-\varphi(w)}{z-w}.
\end{equation*}

\begin{cor}\label{cor: one marked boundary}
    Fix $\gamma \in (0,2]$. Let $h$ be the LGF on $\m S^1$ and $\mathfrak h(z) := - \alpha \log |z - 1|$ where $\alpha$ is a positive constant. Suppose $\varphi \in \Homeo_+(\m S^1)$ fixes 1. Then, the law of the normalized pullback measure $\varphi^* \widehat{\mc M}_{h+\mathfrak h}$ is absolutely continuous with respect to the law of $\widehat{\mc M}_{h+\mathfrak h}$ if and only if $\varphi \in \WP(\m S^1)$ and $u_\varphi(\cdot,1) \in H^{1/2}(\m S^1)$.
\end{cor}

\begin{proof}
    Suppose $\varphi \in \WP(\m S^1)$ and $u_\varphi(\cdot,1) \in H^{1/2}$, in which case $\mathrm{Re}\, u_\varphi(\cdot,1) = \log|(\varphi(\cdot)-1)/(\cdot-1)| \in H^{1/2}$. For each fixed $C$, since the law of $h + (\mathfrak h \wedge C)$ is absolutely continuous to that of $h$, we almost surely have $\varphi^* \mc M_{h + (\mathfrak h \wedge C)} = \mc M_{h \circ \varphi + (\mathfrak h \circ \varphi \wedge C) + Q\log|\varphi'|}$ by Proposition~\ref{prop:quasi-invariance of weighted Liouville measure}. Letting $C\to \infty$, since GMC is a local functional of the field, we obtain $\varphi^* \mc M_{h + \mathfrak h} = \mc M_{h \circ \varphi + \mathfrak h \circ \varphi + Q\log|\varphi'|}$ almost surely. Since $\mathfrak h \circ \varphi - \mathfrak h + Q\log|\varphi'|\in H^{1/2}$ by our assumption that $\mathrm{Re}\, u_\varphi(\cdot,1) \in H^{1/2}$, the law of $\Pi(\varphi) h + \pi_0(\mathfrak h \circ \varphi + Q\log|\varphi'|)$ is absolutely continuous with respect to that of $h + \pi_0(\mathfrak h)$ by Theorem \ref{thm:main_1}. This implies that the law of $\varphi^*\widehat{\mc M}_{h + \mathfrak h}$ is absolutely continuous with respect to that of $\widehat{\mc M}_{h +\mathfrak h}$.

    On the other hand, suppose that the law of $\varphi^*\widehat{\mc M}_{h + \mathfrak h}$ is absolutely continuous with respect to that of $\widehat{\mc M}_{h +\mathfrak h}$ for some $\varphi \in \Homeo_+(\m S^1)$. Pick $\varphi_n \in \Diff_+(\m S^1)$ which converges uniformly to $\varphi$. As in the proof of Proposition~\ref{prop:quasi-invariance of weighted Liouville measure}, there exists a random distribution $\tilde h$ on $\m S^1$ whose law is mutually absolutely continuous with respect to that of the LGF on $\m S^1$ such that $\tilde h_n:= \Pi(\varphi_n) h + \pi_0(\mathfrak h \circ \varphi_n - \mathfrak h + Q\log|\varphi_n'|)$ converges in law to $\tilde h$ as $n\to \infty$. Noting that $\mathrm{Cov}(\langle \tilde h_n,f\rangle, \langle \tilde h_n, g\rangle) = \langle\Pi(\varphi_n)^* f, \Pi(\varphi_k)^* g\rangle$ for any $f,g\in \mc H_0$, a proof analogous to that of Proposition~\ref{prop:quasi-invariance of weighted Liouville measure} gives $\varphi \in \WP(\m S^1)$. Then, $\tilde h_n$ converges in law to $\Pi(\varphi) h + \pi_0(-\alpha\,\mathrm{Re}\, u_\varphi(\cdot,1) + Q\log|\varphi'|)$, and the absolute continuity of its law with respect to LGF implies $u_\varphi (\cdot,1)\in H^{1/2}(\m S^1)$.
\end{proof}

\section{Quasi-invariance of SLE welding}\label{section: main result}
In this section, we introduce the space $\CR(\m S^1)$ of  conformally removable weldings. We endow it with a topology so that composition gives a continuous group action of quasisymmetric circle homeomorphisms $\QS(\m S^1)$ on $\CR(\m S^1)$ (Lemma \ref{lem:removability is QS-invariant under composition}). This is due to the definition of quasi-invariance. For a measure $\ms P$ on a measurable space $(S,\mc F)$, let $G$ be a group that acts (left or right) measurably on $S$. We say that $\ms P$ is \textit{quasi-invariant} under the action of $G$ if for each $g \in G$ and $A \in \mc F$, we have $\ms P(A)=0$ if and only if $\ms P(gA)=0$. That is, the pull-back of $\ms P$ by $g$ is mutually absolutely continuous with respect to $\ms P$. We refer the reader to \cite{quasiinvariant_book} for more background. In particular, we show the measurability of the group action of Weil--Petersson homeomorphisms on SLE weldings.

After that, we introduce the SLE loop shape measure and find the law of the corresponding welding homeomorphism (see Lemma~\ref{lem:welding-homeo}), which leads to the proof of Theorem~\ref{thm:main}.
\subsection{Conformal welding of Jordan curves}\label{section:removability}
      We say that a compact set $K$ is \textit{(quasi)conformally removable} if any homeomorphism of $\hat{\m C}$ that is (quasi)conformal off $K$ is (quasi)conformal on $\hat{\m C}$. 
     An easy application of the measurable Riemann mapping theorem shows that $K$ is conformally removable if and only if it is quasiconformally removable. We refer the reader to the survey \cite{You15} for an in-depth look at conformal removability.

It is well known that if a Jordan curve $\eta$ is conformally removable, then the corresponding welding homeomorphism is unique up to pre- and post-compositions by $\mob(\m S^1)$. Given an oriented Jordan curve $\eta$, let $f:\m D \to \O$ and $g:\m D^* \to \O^*$ be any conformal maps onto the components of $\Chat \setminus \eta$ on the left and the right of the curve, respectively. Suppose $\tilde f:\m D \to \tilde \O$ and $\tilde g: \m D \to \tilde \O^*$ is another pair of conformal maps such that $\tilde \eta:= \Chat \setminus (\tilde \O \cup \tilde \O^*)$ is a Jordan curve and $(\tilde g^{-1} \circ \tilde f)|_{\m S^1} = (g^{-1}\circ f)|_{\m S^1}$. Then, we can define a homeomorphism of $\Chat$ given by $\tilde f \circ f^{-1}$ on $\overline{\O}$ and $\tilde g \circ g^{-1}$ on $\overline{\O^*}$. Since this homeomorphism is conformal on $\Chat \setminus \eta$, it must be conformal on all of $\Chat$. Hence, $\tilde \eta$ is an image of $\eta$ under a M\"obius transformation as desired. However, there is no analytic characterization of weldings corresponding to conformally removable curves \cite{CR_hard}.  

Recall that For a compact set $K$ in $\hat{\m C}$, \textit{Hausdorff dimension} $d_H(K)$ is defined to be 
 \begin{equation*}
    \inf\{\e \geq 0| \inf_{D \in \ms D } \sum_{\mr a \in \mc A} r_\mr a^\epsilon=0\},
 \end{equation*}
 where $\ms D$ is the family of finite coverings $D=\{D_{\mr a}\}_{\mr a \in \mc A}$ indexed by a set $\mc A$ of $K$ by discs $D_\mr a$ with radius $r_\mr a$.  
 \begin{lem} \label{lem:removability is QS-invariant under composition}
    Let $\CR(\m S^1)$ denote the space of orientation-preserving circle homeomorphisms that are weldings of conformally removable Jordan curves. If $\varphi \in \QS(\m S^1)$ and $\psi \in \CR(\m S^1)$, then $\varphi^{-1} \circ \psi$ are $\psi \circ \varphi$ are in $\CR(\m S^1)$. If we further assume that $\varphi \in \Sy (\m S^1) $, then the three curves that solve the welding problem for $\psi$, $\psi\circ\varphi$ and $\varphi^{-1} \circ\psi$ share the same Hausdorff dimension.
  \end{lem}
   \begin{proof}
      Given $\varphi \in \QS(\m S^1)$, choose a quasiconformal extension $\o$ of $\varphi^{-1}$ to $\Chat$ whose Beltrami coefficient $\mu={\partial_{\bar{z}} \omega}/{\partial_z \omega }$ satisifies the symmetry condition $\mu(z) = (z/\bar z)^2\bar \mu(1/\bar z)$.
 For a conformally removable Jordan curve $\eta$ solving the welding problem for $\psi$, let $f : \m D \rightarrow \Omega$ and $g : \m D^* \rightarrow \Omega^*$ denote the conformal maps onto the bounded and unbounded complementary components of $\eta$, respectively. Let $\tilde \o$ be a quasiconformal homeomorphism of $\Chat$ that fixes $0$, $1$ and $\infty$, whose Beltrami coefficient is given by
 \begin{equation} \label{eq:group action Beltrami}
    \tilde{\mu}=
\begin{cases}
\left(  \mu {{f'}}/{\bar{f'}}  \right) \circ f^{-1} & ~\text{in}~ \Omega,\\
~0  & ~\text{otherwise}.
\end{cases}
 \end{equation}
 We define as above so that $\tilde{f}=\tilde{\omega} \circ f \circ \omega^{-1}: \m D \rightarrow \tilde{\omega}(\Omega) $ and $\tilde{g}=\tilde{\omega}\circ g : \m D^* \rightarrow \tilde{\omega}(\Omega^*) $  are exactly the conformal maps corresponding to the Jordan curve $\tilde{\eta}=\tilde{\omega}(\eta)$, whose welding homeomorphism is $\tilde{\psi} := (\tilde g^{-1} \circ \tilde f)|_{\m S^1} =\psi \circ \varphi$. To see that $\tilde{\eta}$ is quasiconformally removable, note that given a homeomorphism $F$ of $\hat{\m C}$ which is quasiconformal off $\tilde{\eta}$, we have that $\tilde{F}= F \circ \tilde{\omega}$ is a homeomorphism of $\hat{\m C}$ which is quasiconformal off $\eta$. Since $\eta$ is (quasi)conformally removable, $F$ is quasiconformal on $\hat{\m C}$ and it follows that $\tilde{F}$ is quasiconformal on $\hat{\m C}$. This shows that if $\psi\in \CR (\m S^1)$ and $\varphi \in \QS(\m S^1)$, then $\psi \circ \varphi \in \CR(\m S^1)$. 

 Let us further assume that $\varphi \in \Sy(\m S^1)$. Let $d$ and $\tilde{d}$ denote the Hausdorff dimensions of $\eta$ and $\tilde{\eta} = \tilde \o(\eta)$, respectively. From classical distortion estimates for quasiconforaml maps \cite{Ast94_hausdorff_dimension_area}, we have
 \begin{equation*}
    ( 1/d-1/2)/K \leq 1/\tilde{d}-1/2  \leq K( 1/d-1/2)
 \end{equation*}
where $K$ is the supremum of dilatation of $\tilde{\omega}$ near $\eta$. This is asymptotically equal to $1$, as $\tilde{\omega}$ is asymptotically conformal near $\eta$. More precisely, we have $K_r \rightarrow 1$ as $r \uparrow 1$, where $K_r$ is the supremum of dilatation of $\tilde{\omega}$ restricted on $\m C \backslash f(r\m D)$.  Therefore, we have $\tilde{d}=d$. 

The analogous results for $\varphi^{-1} \circ \psi$ follow in a similar manner, taking $\tilde \o$ to be a quasiconformal homeomorphism of $\Chat$ where its Beltrami coefficient agrees with the pullback of $\mu$ by $g$ in $\Omega^*$ and $0$ otherwise. 
\end{proof}

 \begin{remark}\label{rem:more fractional dimension}
    We may replace the Hausdorff dimension in Lemma \ref{lem:removability is QS-invariant under composition} by the upper Minkowski dimension, the packing dimension, Assouad dimension, etc., as they share the similar dilatation-dependent distortion bounds. See \cite{CT23_Assouad_Minkowski_dimension_distortion} and the discussion therein. 
 \end{remark}
 Now we define the post-composition $R:\CR(\m S^1)\times\QS(\m S^1)\rightarrow \CR(\m S^1)$ by $R(\psi,\varphi)=\psi\circ \varphi$ and the pre-composition $L:\CR(\m S^1)\times\QS(\m S^1)\rightarrow \CR(\m S^1)$ by $R(\psi,\varphi)=\varphi^{-1}\circ \psi$. Recall that we the topology on $\QS(\m S^1)$ is induced from the Teichm\"uller distance. We define the topology on $\CR(\m S^1)$ by the Carath\'eodory topology on the two conformal maps via conformal welding. More precisely, we say $\psi_n\rightarrow\psi$ if there exist welding solutions $\psi_n =g_n^{-1}\circ f_n$ and $\psi=g^{-1}\circ f$ such that $f_n \rightarrow f$ and $g_n \rightarrow g $ uniformly on compact subsets.
\begin{prop}\label{continuous group action}
    The post-composition and pre-composition are continuous.
\end{prop}
 \begin{proof}
    Suppose $\psi_n \to \psi$ and $\varphi_n\to \varphi$ in $\CR(\m S^1)$ and $\QS(\m S^1)$, respectively. Let $f_n \rightarrow f$ and $g_n \rightarrow g$ uniformly on compact subsets be the conformal maps corresponding to $\psi_n$ and $\psi$. 
    Let $\mu_n$ and $\mu$ be the Beltrami coefficient of the quasiconformal extension $\o_n$ and $\o$ of $\varphi_n^{-1}$ and $\varphi^{-1}$, respectively. We may choose $\mu_n$ so that $\mu_n\rightarrow \mu$ in the $L^\infty$ norm. Define $\tilde \mu_n$ as in \eqref{eq:group action Beltrami} using $\mu_n$ and $f_n$ instead.

      Observe in the proof of Lemma \ref{lem:removability is QS-invariant under composition} that if $f_n \rightarrow f$ uniformly on compact subsets of $\m D$ and $\mu_n \rightarrow \mu$ almost everywhere, then $\tilde{\mu}_n \rightarrow \tilde{\mu}$ almost everywhere on $\Chat$. Hence, if we choose $\o_n$ and $\tilde \o_n$ to fix $0,1,\infty$, then $\o_n^{-1} \to \o^{-1}$ and $\tilde \o_n \to \tilde \o$ uniformly with respect to the spherical metric \cite[Thm~4.6]{lehto2012univalent}. Thus, $\tilde f_n=\tilde \o_n \circ f_n \circ \o_n^{-1} \to \tilde \o \circ f \circ \o^{-1}=\tilde f$ and $\tilde g_n=\tilde \o_n \circ g_n  \to \tilde \o \circ g=\tilde g$ uniformly on compact subsets. 
      
      The analogous results for $\varphi^{-1} \circ \psi$ follow in a similar manner as before.
 \end{proof}
As we will see in the next subsection, from the perspective of the SLE$_\kappa$ loop measure, it is more appropriate to consider a correspondence between the following spaces.
\begin{itemize}
    \item The space of orientation-presering circle homeomorphisms which arise as weldings of conformally removable Jordan curves that fix 1, denoted $\m S^1\backslash\CR(\m S^1)$.
    \item The space of conformally removable Jordan curves $\eta$ on $\Chat$ separating $0$ from $\infty$ with a unit conformal radius of $\Chat \setminus \eta$ viewed from $0$, denoted $\mc J_\#$.
\end{itemize}

\begin{df}\label{df: unit cr corresp}
Given $\eta \in \mc J_\#$, let $\O$ and $\O^*$ be the bounded and unbounded components of $\Chat \setminus \eta$, respectively. Consider the unique pair of Riemann maps $f:\m D\to \O$ and $g: \m D^* \to \O^*$ satisfying $f(0) =0$, $f'(0)=1$, $g(\infty)=\infty$, and $f(1)=g(1)$. The corresponding element of $\m S^1 \backslash \CR(\m S^1)$ is given by $(g^{-1}\circ f)|_{\m S^1}$. 
\end{df}

Note that this correspondence is one-to-one since the only map $\o \in \mob(\Chat)$ satisfying $\o(0)=0$, $\o'(0)=1$, and $\o(\infty)=\infty$ is the identity. 
We endow $\mc J_\#$ with the Carath\'eodory topology for $\O$ and $\O^*$, which is equivalent to the topology of local uniform convergence for $f$ and $g$. We endow the topology on $\m S^1 \backslash \CR(\m S^1)$ that makes the above correspondence a homeomorphism.

\begin{remark}
   Given a function $\o:\m R_+\to \m R_+$, let $\mc J_\o$ be a subset of $\mc J_\#$ consisting of curves for which the Riemann map $f$ and $g$ as in Definition \ref{df: unit cr corresp} admit $\o$ as a modulus of continuity on the $\m S^1$. Then, the topology on $\mc J_\o$ inherited from $\mc J_\#$ is equivalent to that of uniform convergence of the Riemann map $f$ and $g$, and hence that induced by the Hausdorff distance on compact subsets of $\Chat$. 

   Take $\o(r) = r^\alpha$ for some fixed $\a\in(0,1)$. Then, it is a classic fact that $\mc J_\o$ includes all $K$-quasicircles in $\mc J_\#$ with $K < 1/\a$. Moreover, for $\k_\alpha \in (0,4)$ depending on $\a$, the SLE$_\kappa$ shape measure is supported in $\mc J_{\o}$ for $\k \in (0,\k_\a)$. (See the next subsection.)
\end{remark}

  The correspondence in Definition \ref{df: unit cr corresp}, when restricted to quasisymmetric circle homeomorphisms in $\m S^1 \backslash \CR(\m S^1)$ and quasicircles in $\mc J_\#$, agrees with the models of the universal Teichm\"uller curve $\mc T(1)$ as discussed in \cite[Sec.~I.1.2]{TT06}. Let us consider $\m S^1 \backslash \QS(\m S^1)$ as a right topological group \cite{Ber73} equipped with the topology inherited from $\mc T(1)$: i.e., the one induced from the Teichm\"uller distance.
\begin{prop}\label{continuous group action special point}
    The post-composition and pre-composition are continuous restricted on $\m S^1\backslash \CR(\m S^1) \times \m S^1 \backslash \QS(\m S^1)$.
\end{prop}
  \begin{proof}
      We only need to take care of the normalization. Note that if $\tilde f_n$ and $\tilde g_n$ are the conformal maps corresponding to $\psi_n \circ \varphi_n$ as in Definition~\ref{df: unit cr corresp}, then they differ from $\tilde \o_n \circ f\circ \o_n^{-1}$ and $\tilde \o_n \circ g$ by a post-composition by $z\mapsto c_nz$ where $1/c_n = (\tilde \o_n \circ f \circ \o_n^{-1})'(0)$. Since $(\tilde \o_n \circ f \circ \o_n^{-1})'(0) \to (\tilde \o \circ f \circ \o^{-1})'(0)$, we conclude that $\tilde f_n \to \tilde f$ and $\tilde g_n \rightarrow \tilde g$ uniformly on compact subsets. The analogous results for $\varphi^{-1} \circ \psi$ follow in a similar manner as before.
  \end{proof}
  Some homeomorphisms are not proven to be weldings, where we consider the measure on homeomorphisms and endow the topology of uniform convergence. Then the post-composition and pre-composition are still continuous.
\subsection{SLE loop welding measure} 
For $\kappa \in (0,4]$, the $\SLE_\kappa$ loop measure $\mu^{\kappa}$ is a $\sigma$-finite measure on the space of Jordan curves on $\widehat{\m C}$ satisfying the following properties as defined by Kontsevich and Suhov \cite{KS07}.
\begin{itemize}
    \item Generalized restriction property: For any simply connected domain $D \subsetneqq \Chat$, we define $\mu_D^\k$ by
    \begin{equation*}
        \frac{\dd \mu_D^\kappa}{\dd \mu^\kappa}(\cdot) = \mathbf{1}_{\{\cdot \subset D\}}\exp\left( c(\kappa)\Lambda^*(\cdot,D^c)/2\right)
    \end{equation*}
    where $c(\kappa)=(6-\kappa)(3\kappa-8)/2\kappa$ is the central charge of SLE$_\k$ and $\Lambda^*(\eta,D^c)$ is the size of the normalized Brownian loop measure for loops hitting both $\eta$ and $D^c$.
    \item Conformal invariance: If $f:D \rightarrow D'$ is a conformal map between two subdomains of $\widehat{\m C}$, then the pushforward measure $f_*\mu_D^\kappa$ is agrees with $\mu_{D'}^\kappa$.
\end{itemize}

For each $\kappa \in (0,4]$, the SLE$_\kappa$ loop measure $\mu^\kappa$ exists and is unique up to a multiplicative constant. This was proved by Werner \cite{werner_measure} for $\kappa = 8/3$; the measure $\mu^\kappa$ for the entire range of $\kappa \in (0,4]$ was constucted by Zhan \cite{zhan2020sleloop} and was proved to be unique by Bav4rez and Jego \cite{bj_cft_loop}.

From the perspective of conformal welding, it is natural to consider the \textit{shape measure} of the $\SLE_\kappa$ loops. First, consider the restriction $\mu^\kappa_{\m C\setminus \{0\}}$ of the SLE$_\kappa$ loop measure $\mu^\kappa$ to the loops that separate 0 and $\infty$. Given a Jordan curve $\eta$ separating 0 and $\infty$, let $\mathrm{CR}(\eta,0)$ denote the conformal radius $|f'(0)|$ of the bounded component $D_\eta$ of $\widehat{\m C}\setminus \eta$ viewed from 0 where $f:\m D \to D_\eta$ is a conformal map with $f(0) = 0$. 
Then, we define $\mu_{\#}^\kappa$ to be the conditional law\footnote{One way to give this law is to consider the restriction of $\mu_{\m C\setminus \{0\}}^\kappa$ to loops $\eta$ with $1/2\leq\mathrm{CR}(\eta,0) \leq 1$, which is a finite measure since it must wind around the disk $\frac{1}{8} \m D$ and intersect the unit circle $\partial \m D$ \cite{zhan2020sleloop}. Scaling each loop $\eta$ sampled under this measure using the map $z\mapsto z/\mathrm{CR}(\eta,0)$ gives a loop sampled from $\mu_{\#}^\kappa$ by the conformal invariance of the SLE$_\kappa$ loop measure.} of $\mu^\kappa_{\m C\setminus \{0\}}$ on the set of loops $\eta$ for which $\mathrm{CR}(\eta,0) = 1$. We may choose the multiplicative constant for the SLE$_\kappa$ loop measure $\mu^\kappa$ such that the shape measure $\mu_{\#}^\kappa$ is a probability measure; we assume this henceforth.
\begin{df}\label{df:SLE welding}
    For $\kappa \in (0,4]$, we define $\SLE_\kappa^{\mathrm{weld}}$ to be the probability measure on $\m S^1 \backslash \CR(\m S^1)$ given by the pushforward of the SLE$_\kappa$ loop shape measure $\mu_\#^\kappa$ under the measurable one-to-one correspondence $\mc J_\#^{0,\infty} \to \m S^1 \backslash\CR(\m S^1)$ induced by conformal welding. 
\end{df}
\subsection{Conformal welding of quantum disks}
Fix $\gamma \in (0,2)$ and let $\kappa = \gamma^2$.
Building upon Sheffield's quantum zipper \cite{Quantum_zipper}, Ang, Holden, and Sun proved in \cite{AHS23} that conformally welding two independent $\gamma$-LQG disks gives a quantum sphere decorated with an independent SLE$_{\k}$ loop. We give a translation of this result in terms of the welding homeomorphism of two independent LGFs on $\m S^1$.

Recall that given an LGF $h$ on $\m S^1$, we let $\widehat{\mc M}_h $ to be the $\gamma$-GMC measure $\mc M_h = \mc M_h^\g$ corresponding to $h$ normalized to have unit mass. Define $\phi_h=\phi_h^\g\in \Homeo_+(\m S^1)$ by
\begin{equation*}
    \phi_h^\gamma(z):= \exp(2\pi \ii \cdot \widehat {\mc M}_h^\gamma([1,z])),
\end{equation*}
for each $z \in \m S^1$ where $[1,z] \subset \m S^1$ denotes the arc running counterclockwise from 1 to $z$. That is, the pushforward of the normalized GMC measure $\widehat{\mc M}_h$ under the homeomorphism $\phi_h$ is the uniform (arc-length) measure on $\m S^1$.
The following is the main result of this subsection.

\begin{lem}\label{lem:welding-homeo}
    Let $\gamma \in (0,2)$ and $\kappa = \gamma^2 \in (0,4)$. Suppose $h_1,h_2$ are independent LGFs on $\m S^1$ and $\alpha$ is a uniform rotation of $\m S^1$ independent from $h_1$ and $h_2$. Let $\mc W^\kappa$ denote the law of $(\phi_{h_2}^\g)^{-1} \circ \alpha \circ \phi_{h_1}^\g \in \Homeo_+(\m S^1)$. Then, the pushforward of $\mc W^\kappa$ under the natural projection $\Homeo_+(\m S^1) \to \m S^1\backslash\Homeo_+(\m S^1)$ is equivalent to $\SLE_\kappa^{\mathrm{weld}}$. Furthermore, if we identify $\m S^1 \backslash \Homeo_+(\m S^1)$ with the stabilizers of 1 in $\Homeo_+(\m S^1)$, then $\SLE_\kappa^{\mathrm{weld}}$ is equivalent to the law of $(\phi_{h_2 -\gamma\log |\cdot - 1|}^\g)^{-1} \circ \phi_{h_1}^\g$.
\end{lem}
Note that $(\phi_{h_2}^\g)^{-1} \circ \alpha \circ \phi_{h_1}^\g $ is the circle homeomorphism corresponding to the conformal welding of two disks with boundary length measures $\widehat{\mc M}_{h_1}$ and $\widehat{\mc M}_{h_2}$ where we make a random shift for the point on the second disk glued to 1 on the first disk according to the boundary length measure $\widehat{\mc M}_{h_2}$. This result immediately implies Theorem~\ref{thm:main}.

\begin{proof}[Proof of Theorem~\ref{thm:main}]
    Combine Lemma \ref{lem:welding-homeo} with Proposition \ref{prop:quasi-invariance of weighted Liouville measure}.
\end{proof}

We will obtain Lemma~\ref{lem:welding-homeo} using the conformal welding of independent quantum disks each with one interior marked point as established in \cite{ang2024sleloopmeasureliouville}. To state this result, let us recall some basic definitions from the LQG literature. A $\gamma$-LQG disk with one marked point is defined as the equivalence class $(D,\XX,z)/\sim_\gamma$ of the tuple of a domain $D$ conformally equivalent to the unit disk, a random distribution $h$, and a point $z \in D$, where $(D,\XX,z) \sim_\gamma (\tilde D, \tilde \XX, \tilde z)$ if there exists a conformal map $f:D \to \tilde D$ such that 
\begin{equation}\label{eq:lqg-surface-change}
    \XX = \tilde \XX \circ f + Q\log |f'|
\end{equation}
for $Q = \frac{2}{\gamma}+\frac{\gamma}{2}$ and $\tilde z = f(z)$. We define a $\gamma$-LQG sphere with two marked points and a Jordan curve similarly as an equivalence class $(D,\XX,\eta,z_1,z_2)/\sim_\gamma$ up to conformal transformations where the field $\XX$ satisfies the coordinate change rule \eqref{eq:lqg-surface-change}.

We now define Liouville field on $\m D$. While the Liouville theory on simply connected domains with boundary was stated originally on the unit disk $\m D$ \cite{hrv-disk-lqg}, subsequent works such as \cite{fyodorov-bouchaud-proof,RZ_boundary_lcft,fzz_welding} used the upper half-plane $\m H$ as the standard domain. Nevertheless, Proposition~3.4 and Proposition~3.7 of \cite{hrv-disk-lqg} imply that the (infinite) law we give here agrees with the definitions of Liouville fields on $\m H$ up to a $\gamma$-dependent multiplicative factor. See \cite[Section~5]{RZ_boundary_lcft} for further details.

\begin{df}[{\cite[Lem.~4.4--4.5]{fzz_welding}}]\label{df:liouville-field}
    Let $P_{\m D}$ be the law of the free-boundary GFF on $\m D$ normalized to have zero mean on the boundary $\partial \m D$. Let $P_{{\m C}}$ denote the law of the whole-plane GFF with zero mean on the unit circle.\footnote{This is a centered Gaussian field with covariance kernel $G_{{\m C}}(z,w) = -\log|z-w| + \log|z|_+ + \log|w|_+$.} 

    \begin{enumerate}
    \item For $\ell>0$, define $\mathrm{LF}_{\m D}^{(\gamma,0)}(\ell)$ to be the law of the field $\XX^\ell := \hat \XX  + \frac{2}{\gamma} \log \frac{\ell}{\nu_{\hat \XX}(\partial \m D)}$, where
    \begin{equation}
        \hat \XX:=\XX - \gamma \log |\cdot|
    \end{equation} 
    for the field $\XX$ sampled under the reweighted measure $\ell^{-4/\gamma^2}(\nu_{\hat \XX}(\partial \m D))^{4/\gamma^2-1} P_{\m D}$. Let ${\mc M}_{1,0}^{\mathrm{disk}}(\gamma;\ell)$ be the measure on quantum surfaces $(\m D, \XX^\ell, 0)/\sim_\gamma$, where $\XX^\ell$ is sampled from $\mathrm{LF}_{\m D}^{(\gamma,0)}(\ell)$. This is a finite measure on quantum disks with boundary length $\ell$ and one interior marked point. Define 
    \begin{equation}
        \mathrm{LF}_{\m D}^{(\gamma,0)} = \int_0^\infty \mathrm{LF}_{\m D}^{(\gamma,0)}(\ell)\,d\ell \quad \text{and} \quad \mc M_{1,0}^{\mathrm{disk}}(\gamma) = \int_0^\infty \mc M_{1,0}^{\mathrm{disk}}(\gamma;\ell)\,d\ell.
    \end{equation}

    \item For $\ell>0$, define $\mathrm{LF}_{\m D}^{(\gamma,0),(\gamma,1)}(\ell)$ to be the law of the field $\XX^\ell := \hat \XX + \frac{2}{\gamma} \log \frac{\ell}{\nu_{\hat \XX}(\partial \m D)}$, where
    \begin{equation}
        \hat \XX:=\XX - \gamma \log |\cdot| - \gamma \log|\cdot - 1|
    \end{equation} 
    for the field $\XX$ sampled under the reweighted measure $\ell^{-4/\gamma^2+1}(\nu_{\hat \XX}(\partial \m D))^{4/\gamma^2} P_{\m D}$. 
    \item Let $\mathrm{LF}_{{\m C}}^{(\gamma,0),(\gamma,\infty)}$ denote the law of the field $\XX - \gamma \log |\cdot| - 2(Q-\gamma)\log|\cdot|_+ + c$ where $(\XX,c)$ is sampled from the law $P_{{\m C}}\times [e^{2(\gamma-Q)c}dc]$ and $|z|_+ = \max\{1,|z|\}$.
    \end{enumerate}
\end{df}

We note that $\mathrm{LF}_{\m D}^{(\gamma,0)}(\ell)$ and $\mathrm{LF}_{\m D}^{(\gamma,0),(\gamma,1)}(\ell)$ are finite measures for each $\ell>0$ \cite[Cor.~3.10]{hrv-disk-lqg}. See also \cite{RZ_boundary_lcft} for explicit formulas for their total masses. The following lemma clarifies the relation between these two measures, which follows from combining Definition~2.10, Proposition~3.4, and Proposition~3.9 of \cite{fzz_welding} and Proposition~2.21 of \cite{ang2024sleloopmeasureliouville}.

\begin{lem} \label{lem:add-point}
Given a finite measure $\l$, let $\l^\#$ denote the this law normalized to be a probability measure. Fix $\ell>0$. 
    \begin{enumerate}
        \item  Sample a field $\XX$ on $\m D$ from $\mathrm{LF}_{\m D}^{(\gamma,0)}(\ell)^\#$, then sample $e^{\ii \theta} \in \partial \m D$ from the normalized boundary length measure $\hat\nu_{\XX^\ell}:=\nu_{\XX^\ell}/\nu_{\XX^\ell}(\partial \m D)$. Then, $\XX^\ell (e^{\ii \theta}\cdot)$ has the law $\mathrm{LF}_{\m D}^{(\gamma,0),(\gamma,1)}(\ell)^\#$.

        \item Sample a field $\XX$ on $\m D$ from $\mathrm{LF}_{\m D}^{(\gamma,0),(\gamma,1)}(\ell)^\#$ and independently sample a uniform boundary point $e^{\ii \theta} \in \partial \m D$ (from the normalized arc-length measure on $\partial \m D)$. Then, $\XX^\ell (e^{\ii \theta}\cdot)$ has the law $\mathrm{LF}_{\m D}^{(\gamma,0)}(\ell)^\#$.
    \end{enumerate}
\end{lem}

Given quantum disks $\mathcal D_1$ and $\mathcal D_2$, let $\mathrm{Weld}(\mc D_1,\mc D_2)$ denote the \textit{uniform conformal welding} of $\mc D_1$ and $\mc D_2$ as described in \cite[Sec.~2.5]{ang2024sleloopmeasureliouville}. That is, $\mathrm{Weld}(\mc D_1,\mc D_2)$ is a probability measure on the quantum sphere obtained by conformally welding the boundaries of $\mc D_1$ and $\mc D_2$ starting from matching the points sampled independently on each boundary from the normalized boundary length measures. Put otherwise, $\mathrm{Weld}(\mc D_1,\mc D_2)$ is the law of the quantum sphere
$(\widehat {\m C}, \XX, \eta, 0,\infty)/\sim_\gamma$ where, if $D_1$ and $D_2$ are the bounded and unbounded components of $\widehat{\m C} \setminus \eta$, then $(\mc D_1,\mc D_2) = ((D_1,\XX|_{D_1}, 0)/\sim_\gamma, (D_2,\XX|_{D_2},\infty)/\sim_\gamma)$ and the quantum boundary length measures of $\mc D_1$ and $\mc D_2$ agree on $\eta$, which we call the quantum length measure of $\eta$. Moreover, if we sample $p \in \eta$ uniformly from the quantum length measure normalized to be a probability measure, then the joint law of $(D_1,\XX|_{D_1},0,p)/\sim_\gamma$ and $(D_2,\XX|_{D_2},0,p)/\sim_\gamma$ is that of $\mc D_1$ and $\mc D_2$ each with a boundary point sampled independently of each other from the normalized quantum boundary length measure. Let us define the law
\begin{equation}
    \mathrm{Weld}(\mc M_{1,0}^{\mathrm{disk}}(\gamma;\ell), \mc M_{1,0}^{\mathrm{disk}}(\gamma;\ell)) = \int \mathrm{Weld}(\mc D_1,\mc D_2) \, d[\mc M_{1,0}^{\mathrm{disk}}(\gamma;\ell)\times \mc M_{1,0}^{\mathrm{disk}}(\gamma;\ell)](\mc D_1, \mc D_2).
\end{equation}

We now state the conformal welding theorem of \cite{AHS23} for the loops that separate 0 and $\infty$, and translate it to Lemma~\ref{lem:welding-homeo}.

\begin{lem}[{\cite[Lemma~7.7]{ang2024sleloopmeasureliouville}\footnote{In place of $\m C$ and $\mu_\#^\kappa$, \cite{ang2024sleloopmeasureliouville} uses the infinite cylinder $\mc C = \m R\times \m S^1$ and the law of SLE$_\kappa$ loops on $\mc C$ separating $\pm \infty$ that touch $\{0\}\times \m S^1$ from the left but do not cross it, respectively. By the conformal invariance of SLE$_\kappa$ loops and the translation invariance of LF$_{\mc C}^{(\gamma,\pm \infty)}$ \cite[Theorem~2.13]{ahs-integrability-sle}, our statement is equivalent to \cite[Lemma~7.7]{ang2024sleloopmeasureliouville}.}}]\label{lem:acsw}
    Let $\kappa = \gamma^2 \in (0,4)$. If $(\XX,\eta)$ is sampled from $\mathrm{LF}_{{\m C}}^{(\gamma, 0),(\gamma,\infty)} \times \mu_\#^\kappa$, then the law of the loop-decorated quantum surface $(\widehat{\m C}, \XX, \eta, 0,\infty)/\sim_\gamma$ is equal to $\int_0^\infty \ell \cdot \mathrm{Weld}(\mc M_{1,0}^{\mathrm{disk}}(\gamma;\ell), \mc M_{1,0}^{\mathrm{disk}}(\gamma;\ell))\, \dd \ell$ up to a $\gamma$-dependent multiplicative constant.
\end{lem}

\begin{proof}[Proof of Lemma~\ref{lem:welding-homeo}]
    Consider the disintegration 
    \[ \mathrm{LF}_{{\m C}}^{(\gamma, 0),(\gamma,\infty)} \times \mu_\#^\kappa = \int_0^\infty [\mathrm{LF}_{{\m C}}^{(\gamma, 0),(\gamma,\infty)} \times \mu_\#^\kappa](\ell)\,\dd\ell  \]
    where $[\mathrm{LF}_{{\m C}}^{(\gamma, 0),(\gamma,\infty)} \times \mu_\#^\kappa](\ell)$ is a measure on the pair $(\X,\eta)$ such that the quantum length of $\eta$ with respect to $\X$, which we denote $\nu_{\X}(\eta)$, equals $\ell$. From Definition~\ref{df:liouville-field}, we see that a version of $[\mathrm{LF}_{{\m C}}^{(\gamma, 0),(\gamma,\infty)} \times \mu_\#^\kappa](\ell)$ is given by $(\X + \frac{2}{\gamma}\log (\ell/ \nu_{\X}(\eta)),\eta)$ where $(\X,\eta)$ is sampled from the restriction of \[\frac{\ell^{-4/\gamma^2}}{\int_1^\infty t^{-4/\gamma^2}\dd t}\mathrm{LF}_{{\m C}}^{(\gamma, 0),(\gamma,\infty)} \times \mu_\#^\kappa\] to the event $\{\nu_{\X}(\eta)\geq 1\}$. In particular, note that the marginal law of $\eta$ under $[\mathrm{LF}_{{\m C}}^{(\gamma, 0),(\gamma,\infty)} \times \mu_\#^\kappa](\ell)$ is equivalent to $\mu_\#^\kappa$.
    
    Fix $\ell = 1$. Recall that the quantum disk $(\m D, \X,0)/\sim_\gamma$ with $\X$ sampled from $\mathrm{LF}_{\m D}^{(\gamma,0)}(1)^\#$ has the law $\mc M_{1,0}^{\mathrm{disk}}(\gamma;1)^\#$. Define $\phi_\XX(z)=\exp(2\pi \ii\hat\nu_\XX([1,z]))$, where $[1,z]$ denotes the arc from $1$ to $z$ counterclockwise. Suppose $\X_1,\X_2$ are independent fields sampled from $\mathrm{LF}_{\m D}^{(\gamma,0)}(1)^\#$. Observe that if $\alpha_1,\alpha_2$ are independent samples from the uniform measure on the rotation group of $\m S^1$ which are further independent from $\X_1,\X_2$, then $(\alpha_2 \circ \phi_{\X_2|_{\partial \m D}})^{-1} \circ (\alpha_1 \circ \phi_{\X_1|_{\partial \m D}})$ is the homeomorphism corresponding to the uniform conformal welding of quantum disks $\mc D_1 = (\m D, \X_1,0)/\sim_\gamma$ and $\mc D_2 = (\m D, \X_2, 0)/\sim_\gamma$. That is, if $\eta$ is the Jordan curve which corresponds to the welding homeomorphism $(\alpha_2 \circ \phi_{\X_2|_{\partial \m D}})^{-1} \circ (\alpha_1 \circ \phi_{\X_1|_{\partial \m D}})$, then there is a random field $\X$ such that $(\widehat{\m C}, \X, \eta,0,\infty)/\sim_\gamma$ has the law $\mathrm{Weld}(\mc M_{1,0}^{\mathrm{disk}}(\gamma;1)^\#, \mc M_{1,0}^{\mathrm{disk}}(\gamma;1)^\#)$. Hence, by Lemma~\ref{lem:acsw}, $\eta$ has the law $\mu_\#^\kappa$ when considered up to rotations of $\widehat{\m C}$ around 0. By the invariance of the law $\mu_\#^\kappa$ under such rotations, if we choose $\alpha_0 \in \partial \m D$ uniformly yet again independently from $\X_1,\X_2,\alpha_1,\alpha_2$, then $\alpha_0 \cdot \eta$ has the law $\mu_\#^\kappa$. 
    Observe that if we denote the rotation $z\mapsto \alpha_0 \cdot z$ also as $\alpha_0$, then the welding homeomorphism corresponding to $\alpha_0 \cdot \eta$ is given by $(\alpha_2 \circ \phi_{\X_2|_{\partial \m D}})^{-1} \circ (\alpha_1 \circ \phi_{\X_1|_{\partial \m D}}) \circ \alpha_0^{-1}$ when both are considered as random elements of $\m S^1 \backslash \Homeo_+(\m S^1)$. 
    From Definition~\ref{df:liouville-field}, we see that if $\X$ is a sample from $\mathrm{LF}_{\m D}^{(\gamma,0)}(1)^\#$, then the law of $\X|_{\partial \m D}$ is equivalent to that of $h - \frac{2}{\gamma}\nu_h(\m S^1)$ where $h$ an LGF on $\m S^1$. That is, the law of $\phi_{\X|_{\partial \m D}}$ is equivalent to that of $\phi_{h - \frac{2}{\gamma}\nu_h(\m S^1)} = \phi_h$. 
    
    Thus, to paraphrase our conclusion so far, if $h_1,h_2$ are LGFs on $\m S^1$ and $\alpha_0,\alpha_1,\alpha_2$ are uniform random variables on $\m S^1$, all mutually independent, the pushforward of the law of $\phi_{h_2}^{-1} \circ (\alpha_2^{-1} \circ \alpha_1) \circ \phi_{h_1} \circ \alpha_0^{-1}$ under the projection $\Homeo_+(\m S^1) \to \m S^1 \backslash \Homeo_+(\m S^1)$ is equivalent to $\SLE_\kappa^{\mathrm{weld}}$. Since the law of LGF on $\m S^1$ is invariant under fixed rotations of $\m S^1$ and $\alpha_0$ is independent of $h_1$, we have that $\phi_{h_1}(\a_0^{-1}(1))^{-1}\cdot\phi_{h_1} \circ \alpha_0^{-1}$ agrees in law with $\phi_{h_1}$. Moreover, $\alpha_2^{-1} \circ \alpha_1$ is a uniform rotation of $\m S^1$ independent from $h_1$, $h_2$, and $\alpha_0$, so we obtain the first part of the lemma.

    For the second part of the lemma, we observe that $e^{\ii \T}:= (\alpha_2 \circ \phi_{\X_2|_{\partial \m D}})^{-1} \circ (\alpha_1 \circ \phi_{\X_1|_{\partial \m D}})(1) = (\phi_{\X_2|_{\partial \m D}})^{-1}\circ (\alpha_2^{-1} \circ \alpha_1)(1)$ is, conditioned on $\X_2$, a point sampled uniformly from the boundary length measure $\nu_{\X_2}$ normalized to be a probability measure. By Lemma~\ref{lem:add-point}, we see that $e^{-i\T} \cdot (\phi_{\X_2|_{\partial \m D}})^{-1}\circ (\alpha_2^{-1} \circ \alpha_1)$ agrees in law with $\phi_{\tilde \X_2}^{-1}$ where $\tilde \X_2$ is sampled from $\mathrm{LF}_{\m D}^{(\gamma,0),(\gamma,1)}(1)^\#$. Since the law of $\tilde \X_2|_{\partial \m D}$ is equivalent to that of $h_2 - \gamma \log |\cdot - 1|$ from Definition~\ref{df:liouville-field}, the second part of the lemma follows as in the previous paragraph.
    \end{proof}
\subsection{Equivalence for other weldings}
We end with a few remarks on generalizations of Theorem \ref{thm:main}. From Lemma \ref{lem log difference}, for each $\varphi \in \WP(\m S^1)$ and almost every $z_0 \in \m S^1$, we have that 
\begin{equation}
    \hat \varphi(e^{\ii \theta}):= \frac{1}{\varphi(z_0)}\varphi(z_0e^{\ii \theta})
\end{equation}
satisfies 
\begin{equation}
    \mathrm{Re}\,u_{\hat \varphi}(\cdot,1) = \log\abs{\frac{\hat\varphi(\cdot)-\hat\varphi(1)}{\cdot-1}} \in H^{1/2}(\m S^1).
\end{equation}
Hence, the asymmetry in Theorem \ref{thm:main} can be seen to be associated with our choice of stabilizers of 1 as the representatives of $\m S^1\backslash \CR(\m S^1)$. 
We record a symmetric version of Theorem \ref{thm:main} without the ``special'' point 1.
\begin{cor}\label{thm:quasi-invariance of the uniform SLE welding}
    For $\k\in (0,4)$, sample $\psi_\k$ from $\SLE_\kappa^{\mathrm{weld}}$ and $z_0$ from a probability measure that is mutually absolutely continuous with respect to the arc-length measure on the unit circle. Define $\hat\psi_\k(\cdot):=z_0\psi_\k(\cdot)$. Then, for any $\varphi \in \WP(\m S^1)$, the laws of $\hat\psi_\k \circ \varphi $ and $\varphi^{-1} \circ \hat\psi_\k $ are mutually absolutely continuous with respect to that of $\hat \psi_\k$.
\end{cor}

The homeomorphism $\phi_h^\g$ giving the normalized GMC measures $\widehat{\mc M}_h^\g$ of intervals was first introduced in \cite{AJKS}, which showed using analytic methods that if $h_1$ and $h_2$ are independent LGFs on $\m S^1$ and $\g_1,\g_2\geq 0$ are sufficiently small, then $(\phi_{h_1}^{\g_1})\circ (\phi_{h_2}^{\g_2})^{-1}$ is almost surely a welding. (Let $\phi_h^\g$ be the identity map if $\g=0$.) The works \cite{inverse-gmc-welding} or \cite{CWofGMC} showed that $(\phi_{h_1}^{\g_1})^{-1}\circ (\phi_{h_2}^{\g_2})$ is almost surely a welding as well for sufficiently small $\g_1,\g_2$. A notable feature of these works is that the constants $\g_1$ and $\g_2$ can be distinct; however, the law of the Jordan curves that solve the welding problem for these homeomorphisms were not identified.
Our analysis leads to the following quasi-invariance results for these random homeomorphisms.
\begin{cor}\label{cor quasi-invariance of GMC welding}
    Let $\varphi,\tilde \varphi \in \Homeo_+(\m S^1)$ fix $1$. For independent LGFs $h,\tilde h$ on $\m S^1$ and $\g,\tilde \g \in (0,2]$, we have:
    \begin{itemize}
        \item The law of $\phi_h^\g\circ \varphi$ is mutually absolutely continuous with respect to the law of $\phi_h^\g$ if and only if $\varphi \in \WP(\m S^1)$.
        \item The law of $ \varphi\circ (\phi_h^\g)^{-1}$ is mutually absolutely continuous with respect to the law of $(\phi_h^\g)^{-1}$ if and only if $\varphi \in \WP(\m S^1)$.
        \item The law of $ \varphi\circ (\phi_h^\g)^{-1} \circ \phi_{\tilde h}^{\tilde \g} \circ \tilde \varphi$ is mutually absolutely continuous with respect to the law of $(\phi_h^\g)^{-1} \circ \phi_{\tilde h}^{\tilde \g} $ if $\varphi$, $\tilde \varphi \in \WP(\m S^1)$.
    \end{itemize}
\end{cor}

\newpage
\bibliographystyle{alpha}
\bibliography{main}

\end{document}